\documentclass[preprint,12pt,times]{elsarticle}

\usepackage{amssymb,amsmath}
\usepackage{amsfonts,amstext,amsthm,amssymb,comment}
\usepackage{epsfig,graphics,color}

\def\bfs{\boldsymbol}
\def\d{\text{d}}
\newtheorem{lemma}{Lemma}

\numberwithin{equation}{section}

\journal{Journal of Computational and Applied Mathematics}

\begin{document}

\begin{frontmatter}



\title{Application of a conservative, generalized multiscale finite element method to flow models}


\author{Lawrence Bush$^{2,*}$, Victor Ginting$^{2}$, and Michael Presho$^{1}$}

\address{{\it $^1$ Institute for Scientific Computation, Department of Mathematics, Texas A\&M University, College Station, TX 77843.}\\
{\it $^2$ Department of Mathematics, University of Wyoming, Laramie, WY 82071.}\\

\footnotetext{$^*$ Correspondence to: Lawrence Bush, University of Wyoming, Email: lbush4@uwyo.edu} 
}

\begin{abstract}
In this paper we propose a method for the construction of locally conservative flux fields from Generalized Multiscale Finite Element Method (GMsFEM) pressure solutions. The flux values are obtained from an element-based postprocessing procedure in which an independent set of $4 \times 4$ linear systems need to be solved. To test the performance of the method we consider two heterogeneous permeability coefficients and couple the resulting fluxes to a two-phase flow model. The increase in accuracy associated with the computation of the GMsFEM pressure solutions is inherited by the postprocessed flux fields and saturation solutions, and is closely correlated to the size of the reduced-order systems. In particular, the addition of more basis functions to the enriched coarse space yields solutions that more accurately capture the behavior of the fine scale model. A number of numerical examples are offered to validate the performance of the method. 
\end{abstract}


\end{frontmatter}







\section{Introduction} 
\label{intro}
Many physical processes in science and engineering are described by partial differential equations whose coefficients vary over many length scales. Typical examples may include subsurface flows where the permeability of the porous medium is represented by a high-contrast, heterogeneous coefficient. In recent decades, multiscale methods have been introduced as an effective tool for treating these types of problems \cite{akl06,apwy07,eh09,hw97,hfmq98,jlt03}. An important component of this class of methods is the independent construction of a set of multiscale basis functions that span a solution space that is tied to a coarse
grid, i.e., one whose discretization parameter is much larger than the characteristic scale of the heterogeneous
coefficient. In particular, once a precomputed set of basis functions is available, a specified global coupling mechanism may be used in order to obtain the associated coarse scale solution. As the fine scale information is imbedded into the basis functions, a coarse grid solution inherits the fine scale effects of the underlying system.
In other words, the multiscale basis functions offer a direct method of projecting a coarse solution to the fine grid.
The present paper considers a class of multiscale methods that will be used to effectively solve elliptic pressure equations that appear in a two-phase flow model. 

While standard multiscale methods have proven effective for a variety of applications (see, e.g., \cite{eghe06,hw97,jlt03}), we employ a more recent framework in which the coarse space may be systematically enriched so that the approximate solution sought in it converges to the fine grid solution. The enrichment procedure hinges on the construction of localized spectral problems, where dominant eigenfunctions are used in the construction of the enriched space \cite{egw11,ge10part2}. This type of spectral enrichment allows for the number basis functions (and the size of the coarse space) to be flexibly chosen such that a desired level of numerical accuracy may be achieved. This framework, which is coined as the Generalized Multiscale Finite Element Method (GMsFEM), incorporates the enriched solution space  into a continuous Galerkin (CG) global formulation in order to obtain approximate pressure solutions. 

An advantage of employing a continuous Galerkin multiscale formulation is the relative ease of implementation and resemblence to standard finite element variational formulation. However, a well known limitation of CG is that the resulting solution does not satisfy local conservation. In particular, in the cases when it is necessary to couple the resulting fluxes to a transport equation, local conservation is required. While finite volume-type methods, mixed methods, and discontinuous Galerkin methods typically guarantee conservation \cite{akl06,d03,eghe06}, the respective formulations yield systems that are more delicate (and sometimes larger) than the CG counterpart. 
Furthermore, to the best of our knowledge the blend of enrichment techniques with multiscale methods are still
at its infancy as there has not been any attempt to carry out the formulation using other than continuous Galerkin
formulation (see, e.g., \cite{eglp13} for a recent development using discontinuous Galerkin method).
As a result, we consider the alternative of postprocessing a global CG solution in order to obtain the desired conservation. Numerous methods have been proposed in order to postprocess finite element solutions to obtain conservative fluxes (see, e.g., \cite{cgw07,heml00,lrm95,sl09}), however, in this paper we generalize the procedure offered in \cite{bg13} in which the authors perform a global solve and subsequent element-based computations to achieve conservation.

In this paper we propose a technique that provides flux conservation in the context of GMsFEM. For
rectangular finite elements, our method hinges on a postprocessing technique in which independent $4 \times 4$
systems of equations are solved on each coarse element to obtain the conservative fluxes. We note that
similar derivation can be accomplished for triangular finite elements that yields an independent $3 \times 3$ system of equations. While the postprocessing procedure yields conservative fluxes on the coarse scale (which might suffice for some target applications), we also employ an independent downscaling procedure to construct a conservative flux field on the underlying fine grid. We note that coarse scale conservative discontinuous Galerkin GMsFEM formulations have been used in (see, e.g., \cite{eglp13}), yet emphasize that the method proposed in this paper requires no modification to the original CG formulation and allows for the fluxes to be computed on the fine grid. Furthermore, to our knowledge, conservative GMsFEM-type methods have not yet been incorporated for solving multi-phase flow models in the existing literature. To test the performance of the proposed method we solve a standard two-phase flow model using distinct cases of high-contrast permeability coefficients. In all cases, an increase in the dimension of the coarse solution space yields solutions that are shown to more accurately capture the behavior of the fine scale. In particular, the error decline of the elliptic solution (which has been rigorously analyzed in \cite{egw11}), is directly inherited by the resulting flux values and two-phase saturation solutions. 

The rest of the paper is organized as follows. In Sect.~\ref{model} we introduce a standard two-phase model along with a description of the operator splitting technique that is used for solving the model. In Sect.~\ref{sec:gmsfem} we describe the Generalized Multiscale Finite Element Method (GMsFEM), and follow the construction by introducing the procedure for the computation of postprocessed, conservative flux quantities in Sect.~\ref{sec:postprocess}. A variety of numerical tests are offered in Sect.~\ref{numerical} in order to validate the performance the proposed method. To finish the paper we offer some concluding remarks in Sect.~\ref{conclusion}.

\section{Model problem}
\label{model}

\subsection{Two-phase model} \label{sec:twophase}
We consider a heterogeneous oil reservoir which is confined in a domain $\Omega$. The reservoir is equipped with an injection well, from which water is discharged to displace the trapped oil towards the production wells, situated on the perimeter of the domain. 
The dynamics of the movement of the fluids in the reservoir are categorized as an immiscible two-phase system with water and oil (denoted by $w$ and $o$, respectively) that is incompressible. Capillary pressure is not included in the model. Further simplifying assumptions that we use are a gravity-free environment and that the two fluids fill the pore space. Then, the Darcy's law combined with a statement of conservation of mass allow us to write the governing equations of the flow as
\begin{equation}
  \label{eq:p_eqn}
   \nabla \cdot \boldsymbol{v} = q,~~\text{where}~~\boldsymbol{v} = -\lambda(S) k(\boldsymbol{x}) \nabla p, 
\end{equation}
and
\begin{equation}
  \label{eq:sat_eqn}
  \frac{\partial S}{\partial t} + \nabla \cdot (f(S)\boldsymbol{v})=q_w,
\end{equation}
where $\boldsymbol{v}$ is the Darcy velocity, $S$ is the water saturation, and $k$ is the permeability coefficient. The total mobility $\lambda(S)$ and the flux function $f(S)$ are respectively given by:
\begin{equation} \label{eq:flow_params}
   \begin{aligned}
      \lambda(S) &= \frac{k_{rw}(S)}{\mu_w} + \frac{k_{ro}(S)}{\mu_o},~~f(S) &= \frac{k_{rw}(S)/\mu_w}{\lambda(S)},
    \end{aligned}
\end{equation}
where $k_{rj}$, $j=w,o$, is the relative permeability of the phase $j$. 

\subsection{Solution algorithm}
Notice that the elliptic part \eqref{eq:p_eqn} and the transport part \eqref{eq:sat_eqn} of the system are coupled through the total mobility. In order to solve this problem numerically, we use an operator splitting technique \cite{Aziz}, where saturation at the previous time step is used when solving the elliptic part of the system to obtain the velocity $\bfs{v}$. This velocity is then used in an explicit time stepping scheme for the transport equation. This velocity is held constant for a predetermined number of time steps, which yields a new saturation. This new saturation is then used to solve the elliptic problem again and the process is continued until the final time is reached. A schematic of the operator splitting is shown in Fig.~\ref{fig:opsplit}. 
\begin{figure}[htb]
\centering
\includegraphics[scale=0.6]{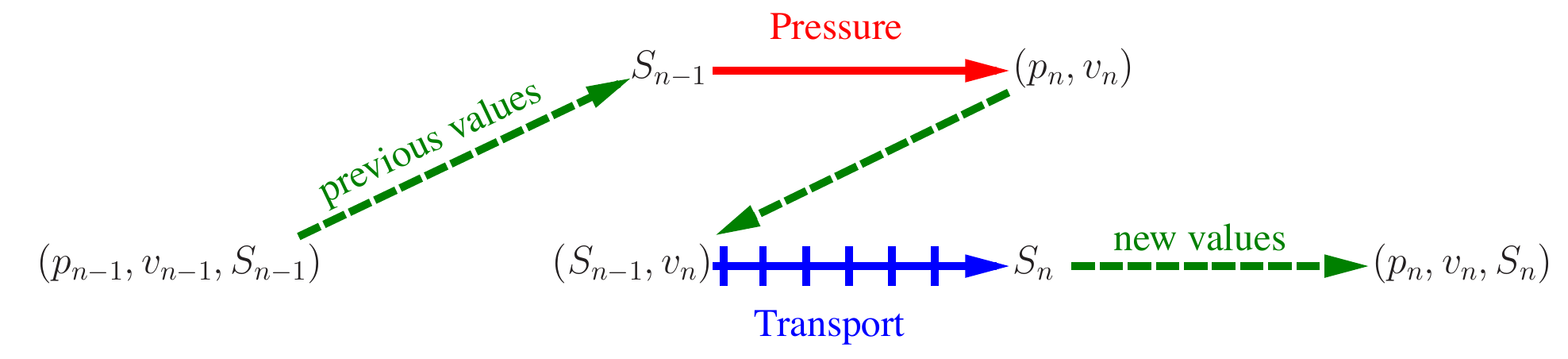} 
\caption{Operator splitting for the two-phase problem.}
\label{fig:opsplit}
\end{figure}

To discretize the transport equation, we first integrate \eqref{eq:sat_eqn} with respect to time and over some
$C_z \subset \Omega$. Here we apply the left end point quadrature rule to its second term in time and use integration by parts to arrive at the approximation
\begin{align*}
\text{meas}(C_z)(S_{z,n} - S_{z,n-1})+ \Delta t \displaystyle\int_{\partial C_z} {\bfs{v}} \cdot \bfs{n} f(S_{z,n-1}) \, \d l = \int_{C_z} q_w \, \d {\bfs x},
\end{align*}
where we have neglected the error terms and 
\begin{align*}
S_{z,n} \approx \frac{1}{\text{meas}(C_z)} \displaystyle\int_{C_z} S(x,t_n) \, \d {\bfs x}.
\end{align*}

The transport part of the system is solved with an explicit time stepping as seen above and an upwinding scheme is used on the term $\int_{\partial C_z} {\bfs{v}} \cdot \bfs{n} f(S) \, \d l$. A review of upwinding on a rectangular mesh can be found for example in \cite{Thomas}. Typically in this situation, it is imperative that numerical approximation
of  ${\bfs v}$ satisfies certain local conservation property. In particular, given ${\bfs v}_h \approx {\bfs v}$, it is
desirable to have
\begin{equation} \label{eq:conserve}
\int_{\partial C_z} {\bfs v}_h \cdot {\bfs n} \, \d l = \int_{C_z} q \, \d {\bfs x}. 
\end{equation}
A natural way of obtaining this local conservation property is to seek the solution
of $(p, {\bfs v})$ from the first order system \eqref{eq:p_eqn}. At the approximation stage, one ends up with the
mixed finite element formulation \cite{Raviart}. A more common approach is to transform \eqref{eq:p_eqn} into a second order equation that governs $p$. The approximate solution for $p$ is sought after which the approximate solution
of ${\bfs v}$ is calculated using the relation $\boldsymbol{v} = -\lambda(S) k(\boldsymbol{x}) \nabla p$,
and hence a {\em postprocessing} procedure is used. Unfortunately, a standard technique such as the Galerkin finite element method
does not allow for a straightforward postprocessing to obtain a ${\bfs v}_h$ that is locally conservative.

Furthermore, aside from the issues pertaining to local conservation, it is almost always impossible to conduct 
numerical simulations at the finer scales if one is to include the ever increasing details of geological information.
A viable alternative is the use of multiscale methods, which is the subject of the next section.

\section{Generalized multiscale finite element method}
\label{sec:gmsfem}
In this section we describe a systematic coarse grid solution technique that may be used as a reduced-order alternative to a standard fine grid approach (such as a fully resolved finite element discretization). The solution procedure is built within the framework of the Generalized Multiscale Finite Element Method (GMsFEM), in which a pressure solution is sought within a space of precomputed multiscale basis functions. This type of generalized method allows for the systematic enrichment of coarse solution spaces based on the underlying structure of the problem (e.g., the structure of the permeability coefficient $k({\bfs x})$).  

To fix our attention, we consider
\begin{equation}\label{eq:elliptic}
\begin{cases}
\begin{aligned}
-\nabla \cdot (\lambda k({\bfs x})  \nabla p)&=q \hspace*{0.5cm} \text{in} ~~ \Omega \subset \mathbb{R}^2 \\
p&= p_{\text{D}} \hspace*{0.5cm} \text{on} ~~ \Gamma_{\text{D}} \\
-\lambda k({\bfs x})  \nabla p  \cdot {\bfs n} &= g_{\text{N}} \hspace*{0.5cm} \text{on} ~~ \Gamma_{\text{N}} 
\end{aligned}
\end{cases}
\end{equation}
with Dirichlet and Neumann boundary conditions given by $p_{\text{D}}, g_{\text{N}}$, respectively, and a forcing function $q$. Within the context of the operator splitting procedure we assume that $\lambda$ is already available.
The variational formulation associated with \eqref{eq:elliptic} is to seek
$p$ with $(p-p_\text{D}) \in H^1_\text{D} =  \{v \in H^1(\Omega) : v|_{\Gamma_{\text{D}}} =0\}$ that satisfies
\begin{equation}\label{eq:wf}
a(p,w)=(q,w)+ \langle g_\text{N}, w \rangle_{\Gamma_\text{N}}~~~\forall ~ w \in H^1_\text{D},
\end{equation}  
where
\begin{equation*}
a(p,w) = \displaystyle\int_\Omega \lambda k({\bfs x})  \nabla p \cdot \nabla w \, \d {\bfs x}, \hspace*{0.2cm}
 (q,w) =  \displaystyle\int_\Omega q w \, \d {\bfs x}, \hspace*{0.2cm} \text{and} \hspace*{0.2cm}
\langle g_\text{N}, w \rangle_{\Gamma_\text{N}} = \displaystyle\int_{\Gamma_{\text{N}}} g_\text{N}w \, \d l.
\end{equation*}

Typical Galerkin finite element methods seek the approximate solution of $p$
in some finite dimensional subspace of $H^1_\text{D}$ that satisfies \eqref{eq:wf}.
This finite dimensional subspace is associated with a discretization of 
$\Omega$ into ${\mathcal  T}_h$ consisting of closed quadrilateral (or triangular) elements $\tau$ such
that  ${\overline \Omega}=\cup_{\tau \in {\mathcal T}_h} \tau$, where
$h=\max_{\tau \in {\mathcal T}_h} \{h_\tau\}$ and $h_\tau$ is the diameter of $\tau$. For example,
by defining  the conforming linear finite element space $\mathcal{V}_h$ over
${\mathcal T}_h$ as
$$ \mathcal{V}_h= \Big \{w_h\in C({{\Omega }}): w_h |_\tau
\hspace{2mm} {\rm  is \hspace{2mm}linear \hspace{2mm}
for\hspace{2mm} all } \hspace{2mm} \tau \in {\mathcal T}_h \hspace{2mm}
{\rm  and} \hspace{2mm}
 w_h |_{\Gamma_{\text{D}}}=0 \Big\},$$
the approximate solution $p_h$ is found to satisfy $(p_h-p_{\text{D},h}) \in \mathcal{V}_h$ and
\begin{equation}\label{eq:stfem}
a(p_h, w_h)=(q,w_h)+ \langle g_\text{N}, w_h \rangle_{\Gamma_\text{N}}~~~\forall ~ w_h  \in \mathcal{V}_h.
\end{equation}  
To proceed, 
we designate $Z$ as the set of all vertices in the discretization of $\Omega$ where $Z = Z_{\text{in}} \cup Z_\text{d} \cup Z_\text{n}$
with $Z_\text{d}$ is the set of vertices on $\Gamma_\text{D}$, $Z_\text{in}$ is the set of interior vertices, and
$Z_\text{n}$ is the set of vertices on $\Gamma_{\text{N}}$.
Designating $\{\phi_\zeta\}_{\zeta \in Z_\text{in} \cup Z_\text{n}}$ as the linear/bilinear basis functions of
$\mathcal{V}_h$, \eqref{eq:stfem} yields a linear system 
\begin{equation} \label{eq:fine}
\mathbf{A}  \mathbf{p} = \mathbf{f},
\end{equation}
where $\mathbf{A}$ is a square matrix with entries $a(\phi_\zeta, \phi_z)$, $\mathbf{f}$
is a vector with entries $(f,\phi_\zeta)$ for $\zeta \in Z_{\text{in}}$ and
$(f,\phi_\zeta) + \langle g_\text{N}, \phi_\zeta \rangle$ for $\zeta \in Z_\text{n}$. The vector $\mathbf{p}$
contains the nodal values of $p_h$, i.e., $p_\zeta = p_h({\bfs x}_\zeta)$, which coincide with the coordinates
in the linear combination of $\{ \phi_\zeta \}$. We note that the size
of this system is $\dim(Z_\text{in} \cup Z_\text{n})$.
\subsection{MsFEM for pressure equations}
The multiscale finite element method (MsFEM) was introduced in \cite{Hou} and
further analyzed in \cite{hwc99}. Similar to the approximation method described earlier,
MsFEM is a continuous Galerkin finite element method that is based on solving \eqref{eq:wf}
in a finite dimensional subspace of $H^1_\text{D}$. Its distinct feature is in the careful
choice of a multiscale finite dimensional subspace that allows calculation
of the approximate solution on $\mathcal{T}_h$ 
without directly resolving the fine scale heterogeneity globally. This means $h$ is much larger
than the characteristic fine scale of $k({\bfs x})$.
Since much of the fine scale heterogeneity has its source from $k({\bfs x})$, this information
should be ingrained in this subspace. This is done through the construction of
the multiscale basis functions that characterize the subspace.

As in the standard Galerkin finite element method, the multiscale basis functions
are associated with the nodes/vertices in the discretization of $\Omega$.  In the interest of offering a straightforward presentation, in what follows we assume that $\mathcal{T}_h$ is a collection of rectangles $\tau$ such as depicted in Fig.~\ref{schematic}. Extension of the method to triangles follows the same line of arguments.

For a vertex $z \in Z_{\text{in}}  \cup Z_\text{n}$, the corresponding multiscale basis functions $\chi_z$ are
defined in such a way that $\chi_{z,\tau} = \chi_z \big |_\tau$ is governed by
\begin{equation} \label{eq:msbasis}
\begin{cases}
\begin{aligned}
  -\nabla \cdot \big( k({\bfs x}) \nabla \chi_{z,\tau} \big) &= 0, ~~\text{in}~~ \tau\\
\chi_{z,\tau}({\bfs x}) &= \phi_{z,\tau}({\bfs x}), ~~\text{on}~~ \partial \tau,
\end{aligned}
\end{cases}
\end{equation}
where
$\phi_{z,\tau}({\bfs x})$ is the standard bilinear function in $\tau$ and $z$ is a vertex of $\tau$.
The multiscale finite dimensional space is defined as
\begin{equation} \label{eq:mspace}
\mathcal{V}_{\text{ms},h} = \text{span}\Big\{ ~\chi_z: z \in Z_{\text{in}}  \cup Z_\text{n} \Big\} \subset H^1_\text{D}.
\end{equation}
The MsFEM solution is $p_{\text{ms},h}$ with $(p_{\text{ms},h} - p_{\text{D},h}) \in \mathcal{V}_{\text{ms},h}$ satisfying
\begin{equation}\label{eq:msfem}
a(p_{\text{ms},h},w_h) = (q,w_h)  + \langle g_\text{N}, w_h \rangle_{\Gamma_\text{N}}~~~\forall ~ w_h  \in \mathcal{V}_{\text{ms},h}.
\end{equation}
Similar to the standard Galerkin finite element method, the linear system resulting from \eqref{eq:msfem}
is
\begin{equation} \label{eq:lsmsfem}
\mathbf{A}_\text{ms}  \mathbf{p} = \mathbf{f}_\text{ms},
\end{equation}
where $\mathbf{A}_\text{ms}$ is a square matrix with entries $a(\chi_\zeta, \chi_z)$, $\mathbf{f}_\text{ms}$
is a vector with entries $(f,\chi_\zeta)$ for $\zeta \in Z_{\text{in}}$ and
$(f,\chi_\zeta) + \langle g_\text{N}, \chi_\zeta \rangle$ for $\zeta \in Z_\text{n}$, and $\mathbf{p}$ is as
in \eqref{eq:fine}.

To this end, we emphasize on the significant advantage of MsFEM. If one is to solve the pressure equation
\eqref{eq:elliptic} to the extent that the fine scale heterogeneity of $k({\bfs x})$ is directly resolved in
the finite element formulation \eqref{eq:stfem}, then ideally the level of mesh resolution on which the flow and
transport are simulated has to be in a comparable scale with the level of subsurface resolution 
exhibited by $k({\bfs x})$. In turn, the resulting linear system such as expressed in \eqref{eq:fine} can have a very 
large dimension whose inversion poses a very challenging task. Employing MsFEM on the other hand, avoids
this drawback as there is no need to pose \eqref{eq:msfem} on a mesh having comparable size to the characteristic
scale of $k({\bfs x})$. This is because the multiscale finite element space in \eqref{eq:mspace} already contains
this information.

Still, potentially there can be a significant error associated with MsFEM approximation that stems
from the imposition of linear boundary condition on $\chi_z$ on $\partial \tau$, which causes a mismatch
as compared to the true solution. This mismatch is even more pronounced when $k({\bfs x})$ exhibits
channelized features and/or scattered inclusions in which the the values of $k({\bfs x})$ are orders of magnitude higher
(or lower) compared to the neighboring regions. Within the context of modeling and simulation of multiphase flow and transport, there have been several research directions aiming toward
alleviating this situation. For example, oversampling techniques \cite{cegh08}, adoption of global information
\cite{eghe06}, and the use of a local-global iterative approach \cite{deg07} have been used for error reduction. More recent investigations involve the alternative of systematically enriching the original coarse space $\mathcal{V}_{\text{ms},h}$. This is the subject of next subsection.

%
\subsection{GMsFEM for pressure equations}
The Generalized Multiscale Finite Element Method (GMsFEM) is based on a systematic enrichment of
$\mathcal{V}_{\text{ms},h}$ \cite{egw11,ge10,ge10part2}. This enrichment is made available by taking
advantage of the knowledge of the spectral properties of the original differential operator governing
the multiscale basis functions $\chi_z$; namely, the left hand side of \eqref{eq:msbasis}.
These types of enriched spaces yield pressure solutions whose errors decrease with respect to the localized eigenvalue behavior. In related work, it is shown that the errors typically depend on the first eigenvalue that is not included in the space construction. We refer the interested reader to \cite{egw11} for rigorous error estimates.

\begin{figure}[htb]
  \begin{center}
      \hspace*{2.1cm} \includegraphics[width=0.8\textwidth]{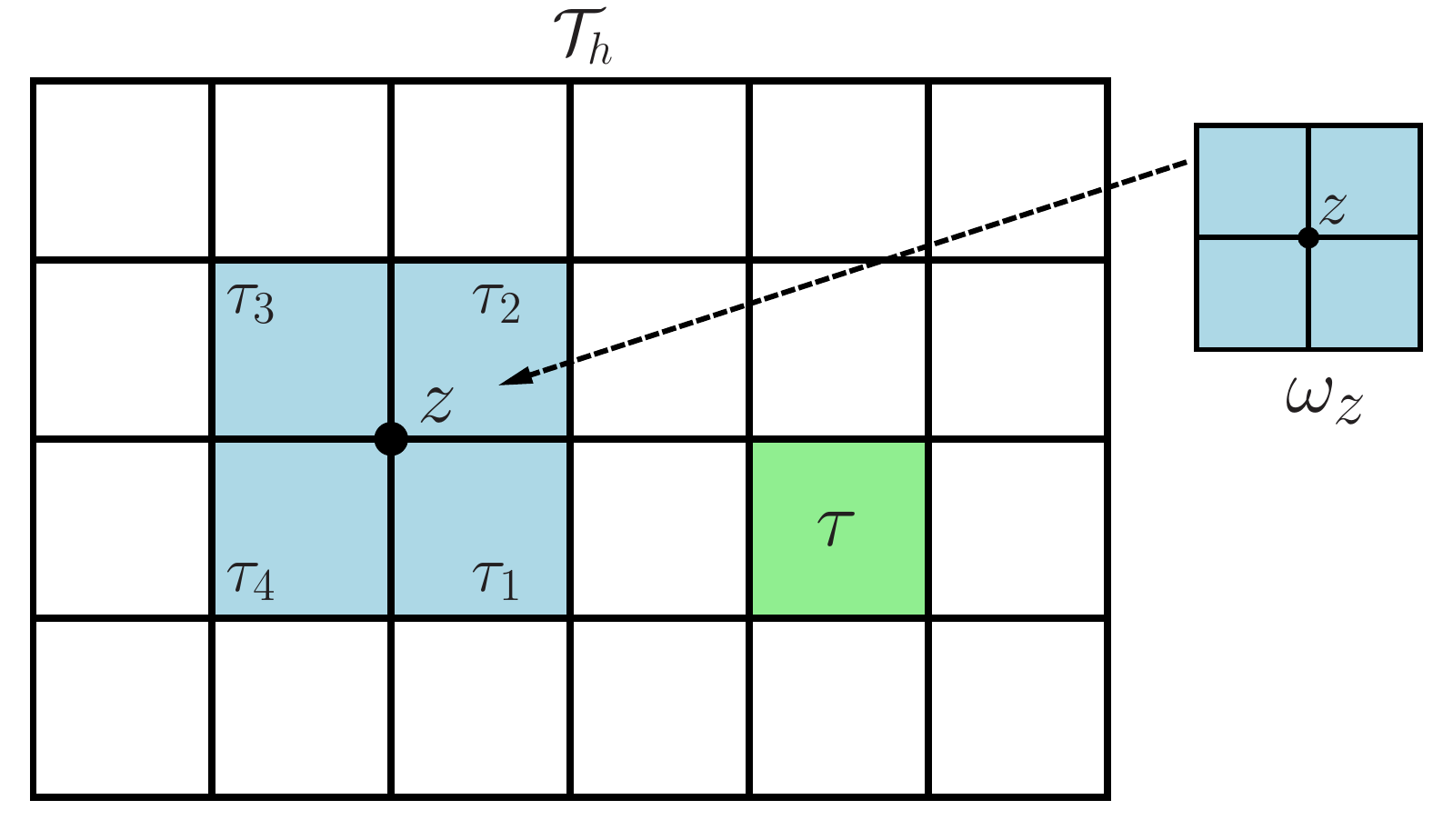}
  \end{center}
  \caption{Discretization of $\Omega$ into $\mathcal{T}_h=\cup \tau$. Here
  $\omega_z = \cup_{j=1}^4 \tau_j$ is the $\text{supp}(\chi_z)$}
  \label{schematic}
\end{figure}

To equip the description below, for a vertex $z \in Z_\text{in} \cup Z_\text{n}$, we let $\omega_z$ be the support of the multiscale basis function
$\chi_z$, namely, $\omega_z = \cup_{j=1}^{4} \tau_j$, where $\tau_j$ are finite elements having
$z$ as one of its vertices. 
Enrichment of $\mathcal{V}_{\text{ms},h}$ employs the pointwise energy of the original multiscale basis functions
\begin{equation}
\widetilde{k} = k \sum_{z \in Z_{\text{in}} \cup Z_\text{n} } h^2 | \nabla \chi_z |^2.
\end{equation}
In particular, using $\widetilde{k}$ as data, we solve an eigenvalue problem
\begin{equation} \label{eigproblem}
\begin{cases}
\begin{aligned}
-\nabla \cdot (k({\bfs x}) \nabla \psi_z) &= \mu_z  \widetilde{k} \psi_z, ~~\text{in}~~ \omega_z\\
-\nabla \psi_z \cdot {\bfs n} &= 0, ~~\text{on}~~\partial \omega_z.
\end{aligned}
\end{cases}
\end{equation}
for each $\omega_z$. We denote the eigenvalues and eigenvectors of \eqref{eigproblem} by $\{ \mu_{z,\ell}  \}$ and
$\{ \psi_{z,\ell} \}$, respectively. By direct observation of \eqref{eigproblem} we see that the first eigenpair is $\mu_{z,1} = 0$ and $\psi_{z,1} = 1$. We order the resulting eigenvalues as
\begin{equation}
\mu_{z,1} \leq \mu_{z,2} \leq \ldots \leq \mu_{z,n} \leq \ldots
\end{equation}
The size of the eigenvalues is closely related to the structure of $\widetilde{k}$ and,
in particular, $m$ inclusions and channels in $\widetilde{k}$ yields $m$ asymptotically vanishing eigenvalues. We use the eigenvectors corresponding to small, asymptotically vanishing eigenvalues for the construction of the
enriched space. In particular, we define the basis functions
\begin{equation} \label{enrichbasis}
\Phi_{z,\ell} = \chi_z \psi_{z,\ell} \quad \text{for} \, \, z \in  Z_{\text{in}} \cup Z_\text{n}~~\text{and} ~~1 \leq \ell \leq L_z,
\end{equation}
\noindent
where $L_z$ denotes the number of eigenvectors that will be chosen for each vertex $z$. We note that
this setting yields $\text{supp}(\Phi_{z,\ell}) = \omega_z$. With the updated multiscale basis functions available, we define the enriched multiscale finite element space as

\begin{equation}
\mathcal{V}_{\text{ems},h} = \text{span}\Big \{ \Phi_{z,\ell}:~~z \in  Z_{\text{in}} \cup Z_\text{n} ~~\text{and} ~~1 \leq 
\ell \leq L_z \Big \} \subset H^1_\text{D}.
\end{equation}
The GMsFEM solution is $p_{\text{ems},h}$ with $(p_{\text{ems},h} - p_{\text{D},h}) \in \mathcal{V}_{\text{ems},h}$ satisfying
\begin{equation}\label{eq:gmsfem}
a(p_{\text{ems},h},w_h) = (q,w_h)  + \langle g_\text{N}, w_h \rangle_{\Gamma_\text{N}}~~~\forall ~ w_h  \in \mathcal{V}_{\text{ems},h}.
\end{equation}

Upon comparison, it is obvious that $\mathcal{V}_{\text{ms},h} \subset  \mathcal{V}_{\text{ems},h}$
and thus $\dim(\mathcal{V}_{\text{ems},h}) > \dim(\mathcal{V}_{\text{ms},h})$. Consequently,
the resulting linear system $\mathbf{A}_\text{ems}  \mathbf{p} = \mathbf{f}_\text{ems}$ has a larger dimension than its counterpart in \eqref{eq:lsmsfem}, where here we view $\mathbf{p}$ as a vector containing
the coordinates (rather than the pressure nodal values) in the linear combination of $\Phi_{z,l}$ that span the
approximate pressure. 
However, this system is still significantly smaller compared
to the fine scale analogue in \eqref{eq:fine} that is solved on mesh that has size of the order of the characteristic scale of
$k({\bfs x})$. Thus, the enriched coarse system offers a suitable reduced-order alternative for obtaining approximate pressure solutions while maintaining an acceptable level of accuracy.

We reiterate that the construction of the enriched basis functions in Eq.~\eqref{enrichbasis} is performed on a respective $\omega_z$ (refer back to Fig.~\ref{schematic}). However, the postprocessing technique offered in the next section is localized to $\tau$. As such, it is important to note that the enriched basis functions need to be
restricted into respective element contributions to implement the proposed procedure.
So while we refer to the enriched basis functions synonymously (whether they are posed on a $\omega_z$ or
$\tau$), we make this distinction for additional clarity.   

\section{Locally conservative flux by postprocessing of the GMsFEM solution}
\label{sec:postprocess}
A pivotal contribution of this section is the introduction of a procedure in which enriched pressure solutions may be postprocessed to ensure that the conservation property in \eqref{eq:conserve} is met. In doing so, we may use the conservative flux quantities that are required to solve the two-phase model described in Sect.~\ref{model}. Additionally, we remark that the conservative flux quantities and associated saturation profiles are shown to inherit the increased level of numerical accuracy that is associated with the enriched coarse space construction.

The postprocessing technique that we propose in order to construct $\widetilde{\bfs{v}}_h \cdot {\bfs n}$ satisfying  \eqref{eq:conserve} is achieved by relegating the evaluation to
independent element-by-element calculations. 
 At this stage, we need to make precise
about what $C_z$ should look like. A suitable choice within the configuration of the nodally based Galerkin
finite element method is to set $C_z$ as the control volume associated
with vertex $z$ (see left plot of Fig.~\ref{fig:taucv}). In this respect, we employ a basic fact about GMsFEM that
we obtain from \eqref{eq:gmsfem}:
\begin{equation} \label{eq:avflux}
\begin{aligned}
a(p_{\text{ems},h},\Phi_{z,\ell}) &= (q,\Phi_{z,\ell}), ~~
\text{for}~~ z \in Z_\text{in}, ~~\ell=1,\cdots,L_z\\
a(p_{\text{ems},h},\Phi_{z,\ell}) &= (q,\Phi_{z,\ell})  + \langle g_\text{N}, \Phi_{z,\ell} \rangle_{\Gamma_\text{N}}, ~~
\text{for}~~ z \in Z_\text{n}, ~~\ell=1,\cdots,L_z.
\end{aligned}
\end{equation}
Since $\text{supp}(\Phi_{z,\ell}) = \omega_z = \cup_{j=1}^N \tau_j$,
\begin{equation}\label{eq:varivertex}
\begin{aligned}
a(p_{\text{ems},h},\Phi_{z,\ell}) &=\int_{\omega_z} \lambda k({\bfs x}) \nabla p_{\text{ems},h} \cdot \nabla
\Phi_{z,\ell} \, \d {\bfs x} =
\sum_{j=1}^{4} Q_{z,\ell,j},\\
(q,\Phi_{z,\ell}) &= \int_{\omega_z} q({\bfs x}) \Phi_{z,\ell} \, \d {\bfs x} = \sum_{j=1}^{4} F_{z,\ell,j},\\
\langle g_\text{N}, \Phi_{z,\ell} \rangle_{\Gamma_\text{N}} &= \int_{\Gamma_\text{N} \cap \partial \omega_z}
g_\text{N} \Phi_{z,\ell} \, \d l = \sum_{j=1}^{4} G_{z,\ell,j},
\end{aligned}
\end{equation}
with
\begin{equation}\label{eq:QFG}
\begin{aligned}
Q_{z,\ell,j} &= \int_{\tau_j}  \lambda k({\bfs x}) \nabla p_{\text{ems},h} \cdot \nabla \Phi_{z,\ell} \, \d {\bfs x},\\
F_{z,\ell,j} &= \int_{\tau_j} q({\bfs x}) \Phi_{z,\ell} \, \d {\bfs x}, ~~~\text{and}~~~
G_{z,\ell,j} = \int_{\Gamma_{\text{N}} \cap \partial \tau_j } g_\text{N} \Phi_{z,\ell} \, \d l.
\end{aligned}
\end{equation}
We conclude for each $\ell =1,\cdots, L_z$ that
\begin{equation} \label{eq:QF}
\sum_{j=1}^{4} (Q_{z,\ell,j} - F_{z,\ell,j})=0, \hspace*{0.2cm} \text{if} ~~ z \in Z_\text{in},~~~
\sum_{j=1}^{4} (Q_{z,\ell,j} - F_{z,\ell,j} - G_{z,\ell,j})=0, \hspace*{0.2cm} \text{if} ~~ z \in Z_\text{n}.
\end{equation}
\begin{figure}[htb]
\centering
\includegraphics[scale=0.7]{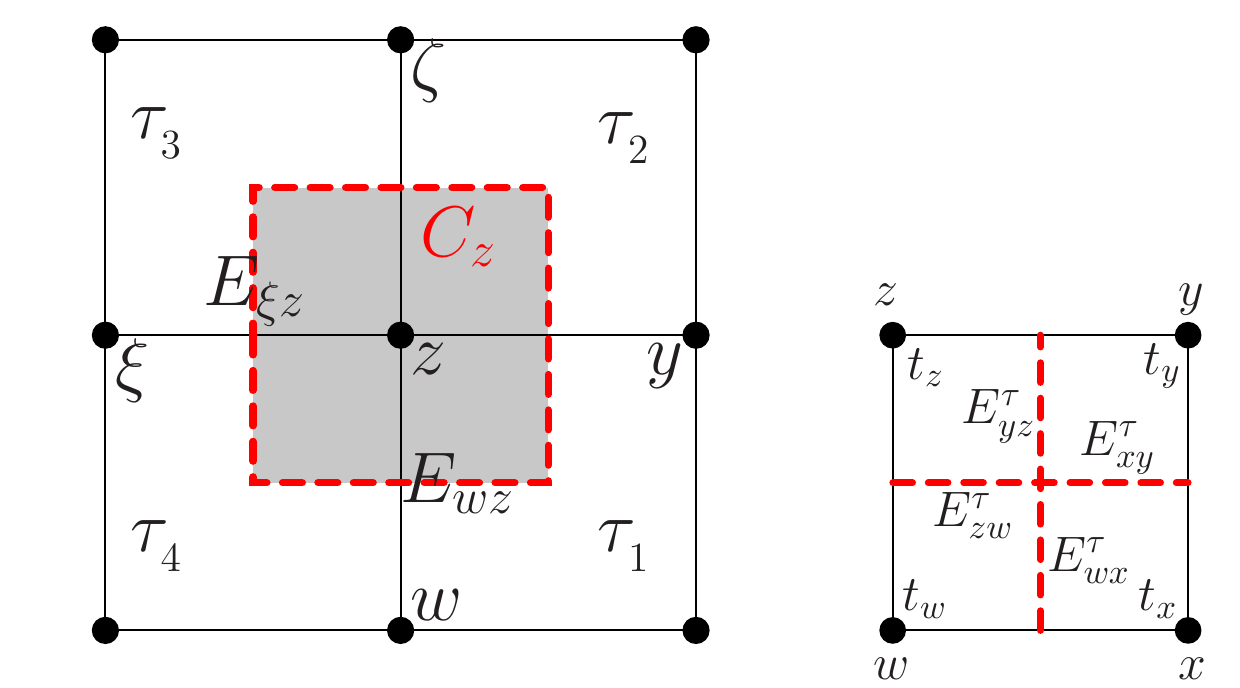}
\caption{Left: $C_z$ is the control volume associated with vertex $z$, where
$\partial C_z =E_{\xi z} \cup E_{zy} \cup E_{wz} \cup E_{z \zeta}$.
Right: A finite element $\tau$
is divided into four quadrilaterals $t_w$, $t_x$, $t_y$, and $t_z$.} \label{fig:taucv}
\end{figure}

\subsection{Auxiliary boundary value problem in $\tau$}
Since we desire independent element-by-element calculations, we proceed with a formulation of
an auxiliary boundary value problem in $\tau$ (see the right plots of
Fig.~\ref{fig:taucv}) having vertices $v(\tau) = \{w,x,y,z\}$.
The boundary value problem governs $\widetilde{p}_\tau$ with
\begin{equation} \label{eq:taubvp}
\begin{cases}
\begin{aligned}
-\nabla \cdot (\lambda k({\bfs x}) \nabla \widetilde{p}_\tau) &= q  \hspace*{0.5cm} \text{in} ~~ \tau\\
-\lambda k({\bfs x}) \nabla \widetilde{p}_\tau \cdot {\bfs n} &= \widetilde{g}_\tau  \hspace*{0.35cm} \text{on} ~~ \partial \tau.
\end{aligned}
\end{cases} 
\end{equation}
Here we designate $\partial \tau = \cup_{\zeta \in v(\tau)} E^\tau_\zeta$, where
$E^\tau_\zeta = \partial \tau \cap \partial t_\zeta$ (i.e. half of each element edge containing the vertex $\zeta$). Furthermore, we set
$\widetilde{g}_\tau$ as piecewise function on $\partial \tau$ such that
\begin{equation}\label{eq:tauedgeflux}
\int_{E^\tau_\zeta} \widetilde{g}_\tau \, \d l = F_{\zeta,1} - Q_{\zeta,1}, \hspace*{0.5cm} \text{for} ~~\zeta \in v(\tau),
\end{equation}
where 
\begin{equation}\label{eq:QFloc}
Q_{\zeta,1} = \displaystyle\int_{\tau} \lambda k \nabla p_{\text{ems},h} \cdot \nabla \Phi_{\zeta,1} \, \d {\bfs x}
 \hspace*{0.3cm} \text{and} \hspace*{0.3cm} F_{\zeta,1} = \int_{\tau} q \Phi_{\zeta,1} \, \d {\bfs x}.
\end{equation}
The existence and uniqueness of the above problem is stated below

\begin{lemma} \label{lem:bvpsol}
The compatibility condition holds for \eqref{eq:taubvp} and thus $\widetilde{p}$ is unique up to a constant. 
\end{lemma}

\begin{proof}
For the compatibility condition \cite{evans}, it suffices to show that
$$
\int_{\partial \tau} \widetilde{g}_\tau \, \d l = \int_\tau q \, \d {\bfs x}.
$$
Because
$\displaystyle \sum_{\zeta \in v(\tau)} \Phi_{\zeta,1} \Big |_\tau = \sum_{\zeta \in v(\tau)} \chi_{\zeta,\tau} =1$,
we get
\begin{align*}
\sum_{\zeta \in v(\tau)} F_{\zeta,1} &= \int_\tau q \sum_{\zeta \in v(\tau)}  \Phi_{\zeta,1} \, \d {\bfs x} =
\int_\tau q \, \d {\bfs x} \\
\sum_{\zeta \in v(\tau)} Q_{\zeta,1} &= \int_\tau \lambda k \nabla p_{\text{ems},h} \cdot 
\nabla \Big(\sum_{\zeta \in v(\tau)}  \Phi_{\zeta,1}\Big) \, \d {\bfs x} = 0.
\end{align*}
Taking into consideration the choice of $\widetilde{g}_\tau$ above, these then give
\begin{equation}\label{eq:compat}
\int_{\partial \tau} \widetilde{g}_\tau \, \d l =  \sum_{\zeta \in v(\tau)} \int_{E^\tau_\zeta} \widetilde{g}_\tau \, \d l
= \sum_{\zeta \in v(\tau)} (F_{\zeta,1} - Q_{\zeta,1})
= \int_\tau q \, \d {\bfs x},
\end{equation}
which is the desired result.
\end{proof}

We note that \eqref{eq:compat} implies
\begin{equation}\label{eq:closeness}
-\displaystyle\int_{\partial \tau} \lambda k \nabla \widetilde{p}_\tau \cdot \bfs{n} \, \d l = -\displaystyle\int_{\partial \tau} \lambda k \nabla p \cdot \bfs{n} \, \d l,
\end{equation}
which shows that the solution of \eqref{eq:taubvp} recovers the flux of $p$ averaged over $\partial \tau$, i.e., a local conservation property.  We use \eqref{eq:taubvp} as a governing principle to derive the postprocessing
technique for calculating a locally conservative flux from $p_{\text{ems},h}$.

\subsection{Elemental calculation}
This elemental calculation is based on discretization of $\tau$ into
quadrilaterals $t_\zeta$, i.e., $\tau = \cup_{\zeta \in v(\tau)} t_\zeta$, each of which yields
$t_\zeta = C_\zeta \cap \tau$, see the right plots of Fig.~\ref{fig:taucv}. 
We set the local solution space as $\mathcal{V}(\tau) = \text{span}\{\Phi_{\zeta,1}\}_{\zeta \in v(\tau)}$. The numerical solution associated with \eqref{eq:taubvp} is to find $\widetilde{p}_{\tau,h } \in \mathcal{V}(\tau)$ satisfying
\begin{equation}\label{eq:fvutildeh}
\displaystyle -\int_{\partial t_\zeta} \lambda k \nabla \widetilde{p}_{\tau,h} \cdot \bfs{n} \, \d l =
\int_{t_\zeta} q \, \d {\bfs x}, \hspace*{0.5cm} \forall \zeta \in v(\tau).
\end{equation}
The following four equations result from \eqref{eq:fvutildeh}:
\begin{equation}\label{eq:loccon}
\begin{aligned}
\phantom{-}q_{wx}^\tau - q_{wz}^\tau &= Q_{w,1} - F_{w,1} + \int_{t_w} q \, \d {\bfs x}, \\
-q_{wx}^\tau + q_{xy}^\tau &= Q_{x,1} - F_{x,1} + \int_{t_x} q \, \d {\bfs x},  \\
-q_{zy}^\tau - q_{xy}^\tau &= Q_{y,1} - F_{y,1} +  \int_{t_y} q \, \d {\bfs x},\\
\phantom{-}q_{zy}^\tau + q_{wz}^\tau &= Q_{z,1} - F_{z,1} + \int_{t_z} q \, \d {\bfs x},
\end{aligned}
\end{equation}
where
\begin{equation*} 
\begin{aligned}
q_{wx}^\tau &= -\int_{E^\tau_{wx}} \lambda k \nabla \widetilde{p}_{\tau,h} \cdot \bfs{n} \, \d l, \hspace*{0.5cm}
q_{zy}^\tau = -\int_{E^\tau_{zy}} \lambda k \nabla \widetilde{p}_{\tau,h} \cdot \bfs{n} \, \d l, \\
q_{wz}^\tau &= -\int_{E^\tau_{wz}} \lambda k \nabla \widetilde{p}_{\tau,h} \cdot \bfs{n} \, \d l, \hspace*{0.5cm}
q_{xy}^\tau = -\int_{E^\tau_{xy}} \lambda k \nabla \widetilde{p}_{\tau,h} \cdot \bfs{n} \, \d l,
\end{aligned}
\end{equation*}
and $E^\tau_{\zeta \eta} = \partial t_\zeta \cap \partial t_\eta$, for $\zeta, \eta = w, x, y, z$, and $\zeta \ne \eta$.
We note that similar equations will be created if $\tau$ is triangle (see \cite{bg13}).

Because $\widetilde{p}_{\tau,h} = \sum_{\zeta \in v(\tau)} \alpha_\zeta \Phi_{\zeta,1}$ with unknown
$\alpha_\zeta$, the above equation yields a linear system of the form $\widetilde{\mathbf{A}}\bfs{\widetilde{\alpha}} = \widetilde{\mathbf{f}}$ where 

\begin{equation}\label{eq:localsys}
\widetilde{\mathbf{A}}_{\zeta \eta} = -\displaystyle\int_{E_{\zeta \eta}} \lambda k\nabla \Phi_{\eta,1} \cdot \bfs{n} \, \d l \hspace*{0.3cm} \text{and} \hspace*{0.3cm} \widetilde{\mathbf{f}}_\zeta = \displaystyle\int_{t_\zeta} q \, \d {\bfs x} - \displaystyle\int_{E_\zeta^\tau} \widetilde{g}_\tau \, \d l.
\end{equation}
Here the system is of dimension 4.
We note that in the case of $\tau$ adjacent to $\Gamma_\text{N}$, a similar linear system can be derived that takes into account the effect of $g_\text{N}$.

Since this linear system is obtained from numerical approximation of \eqref{eq:taubvp}, which is
a Neumann problem, the matrix $\widetilde{\mathbf{A}}$ is singular. On the other hand, because the system has a small dimension, we may specify the value at one of the vertices to remove the null space of 
$\widetilde{\mathbf{A}}$, for instance by adding a constant to one of the entries of  $\widetilde{\mathbf{A}}$,
and then invert the modified matrix. The fact that $\bfs{\widetilde{\alpha}}$ is not unique is irrelevant because
the desired final result from the system is the flux as governed by $q^\tau_{\zeta \eta}$ which is unique.
\subsection{Upscaled local conservation}
The next lemma verifies that the aggregation of elemental calculations satisfies \eqref{eq:conserve}
for $z \in Z_{\text{in}}$. We note that the same is true for $z \in Z_\text{n}$ whose proof we omit for simplicity.

\begin{lemma}\label{lem:conservation}
Fix a $z \in Z_\text{in}$ with $\omega_z = \cup_{j=1}^4 \tau_j$ and control volume $C_z$ (see Fig.~\ref{fig:taucv}).
Set
$$
\begin{aligned}
\int_{E_{\xi z}} \widetilde{\bfs v}_h \cdot {\bfs n} \, {\rm d} l &= q^{\tau_3}_{\xi z} + q^{\tau_4}_{\xi z},~~~
\int_{E_{z y}} \widetilde{\bfs v}_h \cdot {\bfs n} \, {\rm d} l = q^{\tau_1}_{z y} + q^{\tau_2}_{z y},\\
\int_{E_{w z}} \widetilde{\bfs v}_h \cdot {\bfs n} \, {\rm d} l &= q^{\tau_1}_{w z} + q^{\tau_4}_{w z}, ~~~
\int_{E_{z \zeta}} \widetilde{\bfs v}_h \cdot {\bfs n} \, {\rm d} l = q^{\tau_2}_{z \zeta} + q^{\tau_3}_{z \zeta}.
\end{aligned}
$$
Then
$$
\int_{\partial C_z} \widetilde{\bfs v}_h \cdot {\bfs n} \, {\rm d} l = \int_{C_z} q \, {\rm d} {\bfs x}.
$$
\end{lemma}

\begin{proof}
The basis of the proof is using the last equation in  \eqref{eq:loccon} applied to $\tau_j$, for $j=1,\cdots,4$.
We perform a direct calculation and rearrange the terms involved to get
\begin{equation}
\begin{aligned}
\int_{\partial C_z} \widetilde{\bfs v}_h \cdot {\bfs n} \, {\rm d} l 
&= \int_{E_{\xi z}} \widetilde{\bfs v}_h \cdot {\bfs n} \, {\rm d} l +
\int_{E_{z y}} \widetilde{\bfs v}_h \cdot {\bfs n} \, {\rm d} l +
\int_{E_{w z}} \widetilde{\bfs v}_h \cdot {\bfs n} \, {\rm d} l +
\int_{E_{z \zeta}} \widetilde{\bfs v}_h \cdot {\bfs n} \, {\rm d} l\\
&= \big(q^{\tau_3}_{\xi z} + q^{\tau_4}_{\xi z} \big) + \big(q^{\tau_1}_{z y} + q^{\tau_2}_{z y}\big)
+ \big(q^{\tau_1}_{w z} + q^{\tau_4}_{w z} \big) + \big( q^{\tau_2}_{z \zeta} + q^{\tau_3}_{z \zeta} \big)\\
&= \big(q^{\tau_1}_{z y}  + q^{\tau_1}_{w z} \big) + \big( q^{\tau_2}_{z y} + q^{\tau_2}_{z \zeta}  \big)
+ \big( q^{\tau_3}_{\xi z}  +q^{\tau_3}_{z \zeta} \big) + \big( q^{\tau_4}_{\xi z}  + q^{\tau_4}_{w z} \big) \\ 
&= \sum_{j=1}^4 \Big( Q_{z,1,j} - F_{z,1,j} \Big) + \sum_{j=1}^4  \int_{t_{z,j}} q \, \d {\bfs x},
\end{aligned}
\end{equation}
where we have appropriately translated the local indexing in \eqref{eq:loccon} for $Q_{\zeta,1}$ and $F_{\zeta,1}$ into $Q_{z,1,j}$ and $F_{z,1,j}$, respectively.
Notice that $\sum_{j=1}^4 (Q_{z,1,j} - F_{z,1,j})=0$ by \eqref{eq:QF}.
Furthermore, $t_{z,j} = C_z \cap \tau_j$ and thus $\cup_{j=1}^4 t_{z,j} = C_z$. This completes the proof.
\end{proof}

As we see from the above description, what the postprocessing has been able to accomplish is a reconstruction of 
{\em upscaled} (or averaged) conservative flux in $C_z$ whose diameter is on the 
order of $h$. However, MsFEM (and for that matter GMsFEM) always utilizes $h$ that is far greater than the 
characteristic scale of $k({\bfs x})$. In this context, the upscaled flux gathered in Lemma~\ref{lem:conservation}
is comparable to entries of $\mathbf{p}$ coming from solving \eqref{eq:msfem} (respectively from solving \eqref{eq:gmsfem} for GMsFEM). Since $\mathcal{V}_{\text{ms},h}$ and $\mathcal{V}_{\text{ems},h}$
are built from the multiscale basis functions containing fine scale information of $k({\bfs x})$, having this
 $\mathbf{p}$ allows for calculation of the approximate pressure down to the level of characteristic scale
 of $k({\bfs x})$. However, the current stage of development does not yet allow the upscaled flux to have this capability.
 
At many practical levels, obtaining upscaled locally conservative fluxes already allows for simulation of multiphase flow and the
transport problem in Sect.~\ref{model}. However, the transport equation \eqref{eq:sat_eqn} must then be discretized
at the same mesh level as where the approximate pressure is solved, i.e., with $h$ that is far greater than the 
characteristic scale of $k({\bfs x})$. Indeed, this practice is sufficient and suitable for some target applications.
Nonetheless, for a large class of problems, such as modeling flows in channelized subsurface with potential 
localized features, achieving an acceptable accuracy does require the simulation to be performed with a
discretization parameter that is comparable with the scale of $k({\bfs x})$. For this to occur, we need the flux to be 
conservative at that finer  scale. In the next subsection, we offer a downscaling procedure that allows this capability by taking advantage of the upscaled  conservative  flux above.

\subsection{A downscaling procedure}
The main driving fact for the downscaling procedure is the realization that after the postprocessing in
Sect.~\ref{sec:postprocess}, we have
$$
\int_{\partial C_z} \widetilde{\bfs v}_h \cdot {\bfs n} \, {\rm d} l = \int_{C_z} q \, {\rm d} {\bfs x}, ~~\forall~~ C_z,
$$
which can be thought of a statement of as compatibility condition in $C_z$. This means we can proceed with
formulating a boundary problem
\begin{equation} \label{eq:Cvbvp}
\begin{cases}
\begin{aligned}
-\nabla \cdot (\lambda k({\bfs x}) \nabla \widetilde{p}_{_{C_z}}) &= q  \hspace*{0.5cm} \text{in} ~~ C_z\\
-\lambda k({\bfs x}) \nabla \widetilde{p}_{_{C_z}}\cdot {\bfs n} &= \widetilde{\bfs v}_h \cdot {\bfs n} \hspace*{0.35cm} \text{on} ~~ \partial C_z.
\end{aligned}
\end{cases} 
\end{equation}
Here $\widetilde{\bfs v}_h = -\lambda k \nabla \widetilde{p}_{\tau,h}$ that is evaluated pointwise
on segments of $\partial C_z$ that are in $\tau$. In fact, this is the same calculation that we did in order to derive  \eqref{eq:loccon}. In this way the compatibility condition is readily satisfied, and thus the solution of
\eqref{eq:Cvbvp} exists. Similar to the postprocessing in Sect.~\ref{sec:postprocess}, nonuniqueness
of the solution is of no concern since our interest is to gather $-\lambda k \nabla \widetilde{p}_{_{C_z}}$
from \eqref{eq:Cvbvp}.

In our implementation, we numerically solve \eqref{eq:Cvbvp} for every $C_z$ in $\Omega$ with
the discretization of $C_z$ using the same mesh configuration as that of $\tau$ when numerically solving
\eqref{eq:msbasis}. Obviously, the associated mesh parameter should be smaller than $h$ and comparable
to the characteristic scale of $k({\bfs x})$. We use the standard Galerkin finite element method \eqref{eq:stfem} for both of these problems. In addition, the numerical solution of \eqref{eq:Cvbvp} is further postprocessed 
(see \cite{bg13} for the description) to get a locally conservative flux on the finer scale. Once this is done for all
$C_z$, we obtain a the downscaled locally conservative flux in $\Omega$, which is in turn used
in the simulation of transport equation \eqref{eq:sat_eqn}.

We should like to emphasize that the main advantage of the proposed series of upscaled and downscaled
postprocessings is due to their independence of each other. The procedure is immediately parallelizable, fits well in 
the framework of CPU-GPU clusters, while the cost for communication is  minimal. In this respect, they are indeed 
in the same spirit as the multiscale basis functions calculations.

\section{Numerical examples}
\label{numerical}

A variety of numerical examples that validate the performance of the proposed method are presented in this
section. In particular, we perform a convergence study to
address the convergence of the flux to a fine-scale reference solution as the number of enriched basis functions is increased. In general, the improvement of the solution is shown to be dependent on the type of permeability coefficient and the level of enrichment. Applications to single and multi-phase flow are also presented in this section. 

We consider two distinct permeabilities within this section. Their spatial profiles are shown in Fig.~\ref{fig:perm} with the left field shown in physical scale, and the right in log scale. All of the
applications presented in this section use the domain $\Omega = [0, 1] \times [0, 1]$. The left plot is a deterministic, high-contrast coefficient with abrupt transitions between regions of low and high permeability. This permeability is posed on a fine mesh of $100 \times 100$ elements. The right plot in Fig.~\ref{fig:perm} is a single realization of a random, channelized permeability that is posed on a $120 \times 120$ mesh.    
Both examples of permeability exhibit high-contrast features, which can make solving Eq.~\eqref{eq:p_eqn} a demanding task. The ratio between the  maximum ($k_\text{max}$) and minimum value ($k_\text{min}$) of $k(\bfs{x})$ can be thought of as representing the condition number of the resulting linear system in the finite element method. Compared with other methods such as the finite volume element method or mixed finite element method, continuous Galerkin finite element method has an advantage when combined with the postprocessing because the resulting linear systems can be easier to solve \cite{bg13}.

\begin{figure}[t]
  \begin{center}
      \includegraphics[width=1\textwidth]{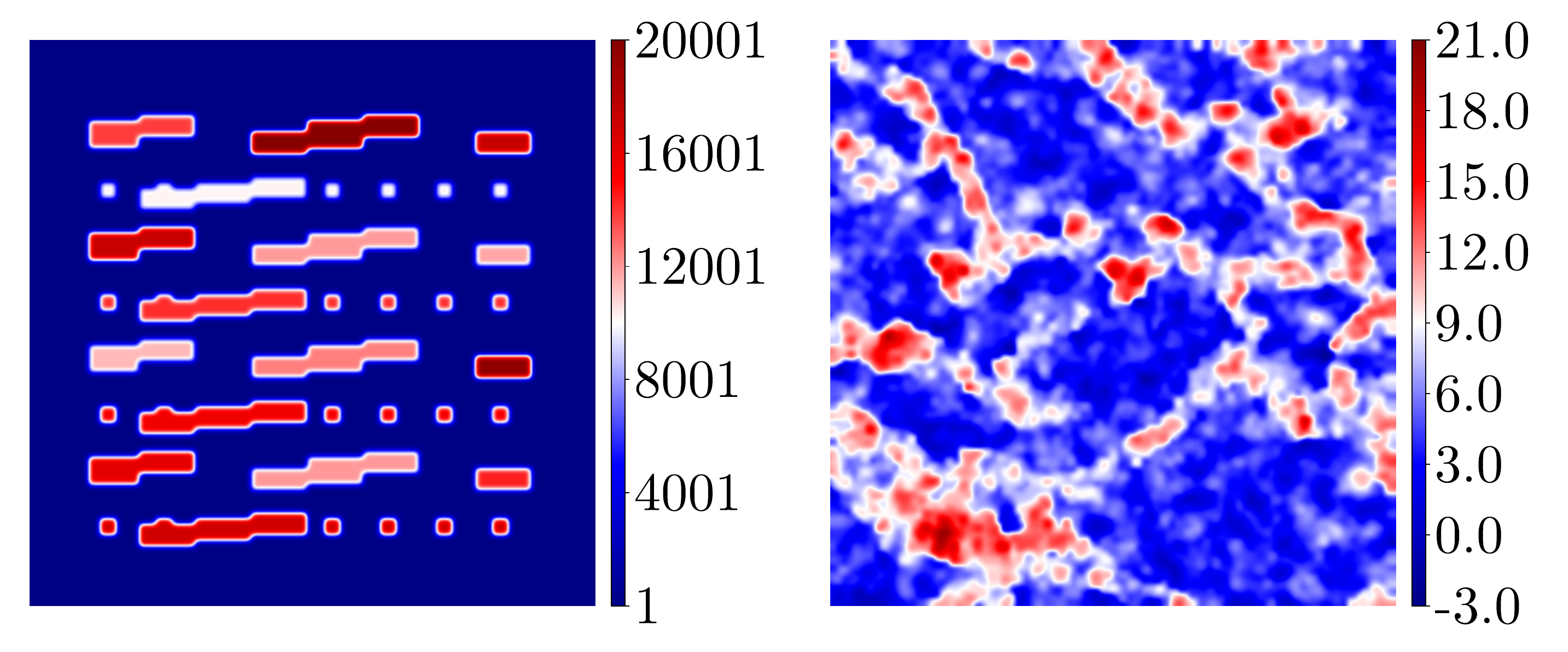}
  \end{center}
  \vspace*{-0.7cm}
  \caption{Permeabilities used in the examples. The plot on the left is shown in physical scale, while the right is shown in log scale with the left being referred to as deterministic
  (with $k_\text{max}/k_\text{min} \approx 2\times 10^4$) and the right as random
  (with $k_\text{max}/k_\text{min} \approx 1.8\times10^6$). }
  \label{fig:perm}
\end{figure}

\subsection{Single-phase flow}
In this subsection, we compare the velocities obtained from postprocessing Eq.~\eqref{eq:p_eqn} to be coupled to the transport problem in Eq.~\eqref{eq:sat_eqn}. To solve the pressure equation in \eqref{eq:p_eqn} we use Dirichlet conditions $p_L = 1$ and $p_R = 0$ on the left and right boundaries of $\Omega$, along with no flow (zero Neumann) conditions on the top and bottom boundaries. We also assume that there is no external forcing, i.e., that $q=0$. Table \ref{tab:fluxerror} shows the relative error of the dowscaled velocities obtained by the postprocessing procedure in Sect.~\ref{sec:postprocess}. We observe that for the deterministic permeability, an appropriate choice of $L_z$ (the number of basis functions chosen for each respective node) significantly reduces the error associated with the downscaled velocities. For the random field, we see that the error reduction is less pronounced, yet that a clear error decline is still evident. This difference is due to the nature of the spectrum associated with the eigenvalue problem in Eq.~\eqref{eigproblem} and, in particular, the eigenvalue behavior corresponding to the deterministic field lends itself to a more pronounced error decline (see, e.g., \cite{egw11}). For this initial set of results all MsFEM/GMsFEM  solutions were computed on $10 \times 10$ coarse mesh.

\begin{table}
\begin{center}
\caption{Comparison of relative velocity errors for different levels of enrichment. The $L^2$ norm of the downscaled velocity was computed for each of the cases below and compared with the fully-resolved reference values. }
\vspace*{0.2cm}
\begin{tabular}{| c | c | c |}
\hline
 & Deterministic & Random  \\ \hline
$L_z=1$ & 1.458 & 0.401    \\ \hline
$L_z=2$ & 0.480 & 0.320     \\ \hline
$L_z=4$ & 0.203 & 0.283     \\ \hline
$L_z=6$ & 0.199 & 0.274    \\ \hline
\end{tabular}\label{tab:fluxerror} 
\end{center}
\end{table}

For a visual representation of the improvement, the computed velocity is plotted on global and localized regions for the deterministic and random permeability fields in Figs.~\ref{fig:quiver_hc} and \ref{fig:quiver_stan}, respectively. In either figure, the left plots show the reference fine scale velocity on the global domain, and on a subregion where the permeability has a large change in value. The center plots show the resulting downscaled velocity on the same regions resulting from the MsFEM (i.e., the case when $L_z=1$). We note that significant differences are noticeable between the reference values and this initial case. Finally, the right plots show the computed velocity corresponding GMsFEM with  $L_z=4$ on the same regions. Figs.~\ref{fig:quiver_hc} and \ref{fig:quiver_stan} clearly illustrate how the enrichment produces a more accurate representation of the velocity on the fine scale. In addition, this increase in accuracy is particularly evident in regions where the permeability contrast is most extreme. We also note that there are some negative values for the horizontal velocity in the random permeability due to the channelized nature of the permeability in some regions.

\begin{figure}[htb]
  \begin{center}
  	  \includegraphics[width=0.9\textwidth]{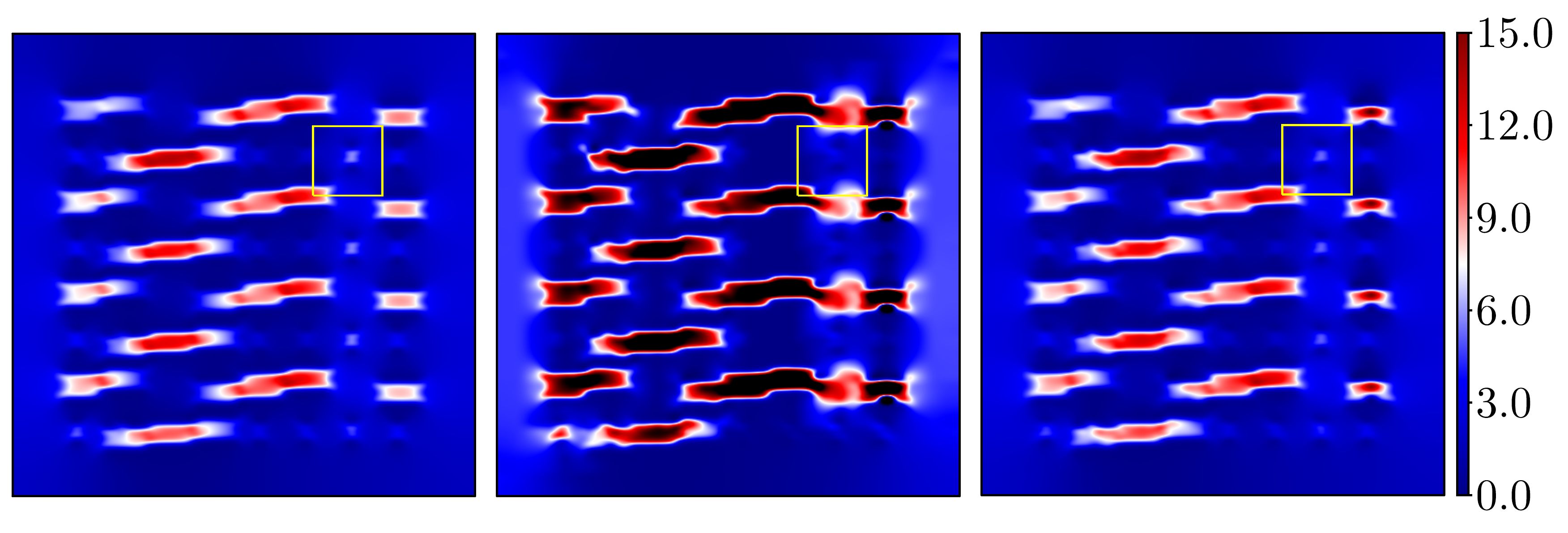}
  	   \includegraphics[width=0.9\textwidth]{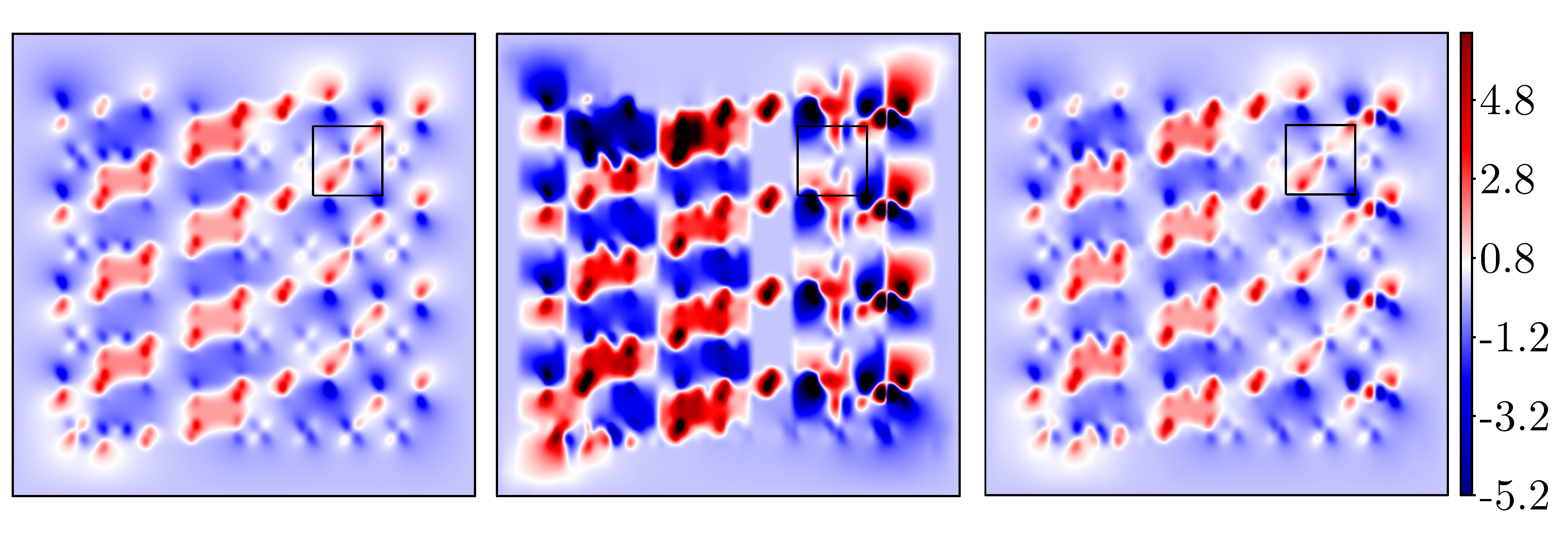}
  	   
  	   \hspace*{-1.0cm}
      \includegraphics[width = 0.25\textwidth, keepaspectratio = true]{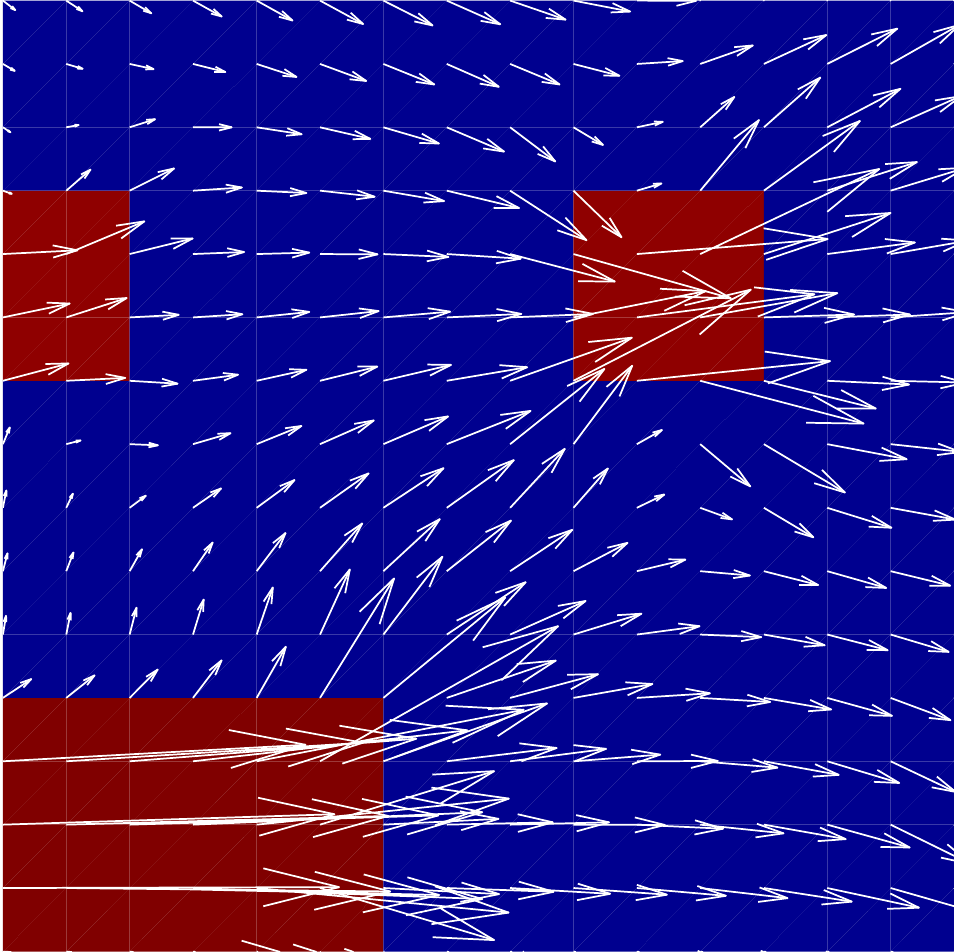} 
      \hspace*{0.3cm}
      \includegraphics[width = 0.25\textwidth, keepaspectratio = true]{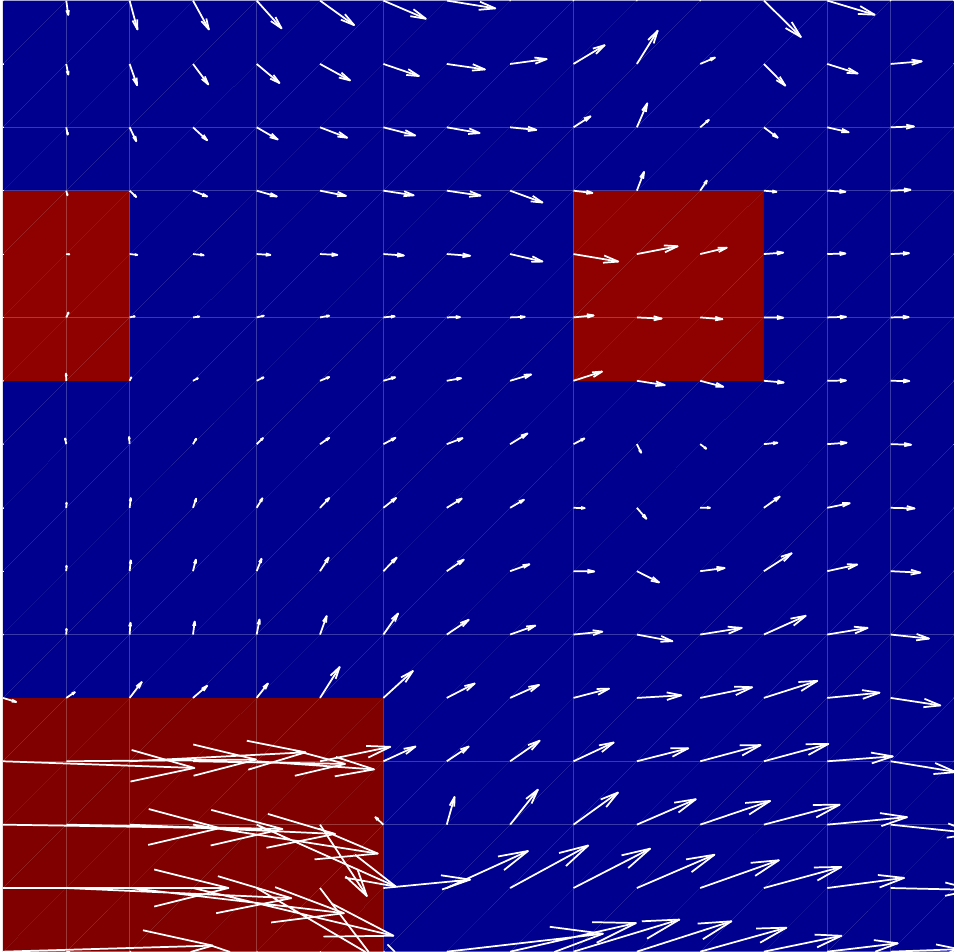} 
      \hspace*{0.3cm}
      \includegraphics[width = 0.25\textwidth, keepaspectratio = true]{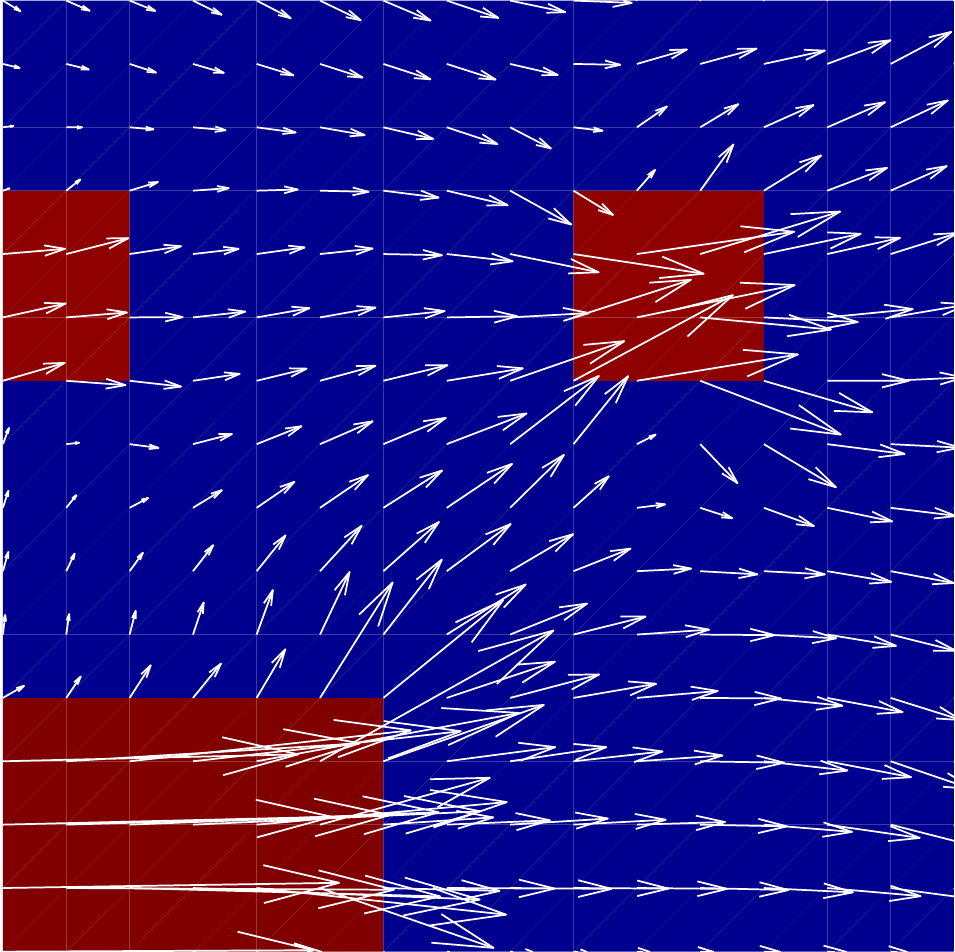}
  \end{center}
  \vspace*{-0.7cm}
  \caption{Velocity computed using three methods. The top two plots show the velocity profile on the whole domain using deterministic permeability with the reference on the left, $L_z=1$ in the center, and $L_z=4$ on the right. The bottom set of plots shows the velocity on the region shown in black and yellow on the top two plots.  }
  \label{fig:quiver_hc}
\end{figure}

\begin{figure}[htb]
  \begin{center}
  \includegraphics[width=0.9\textwidth]{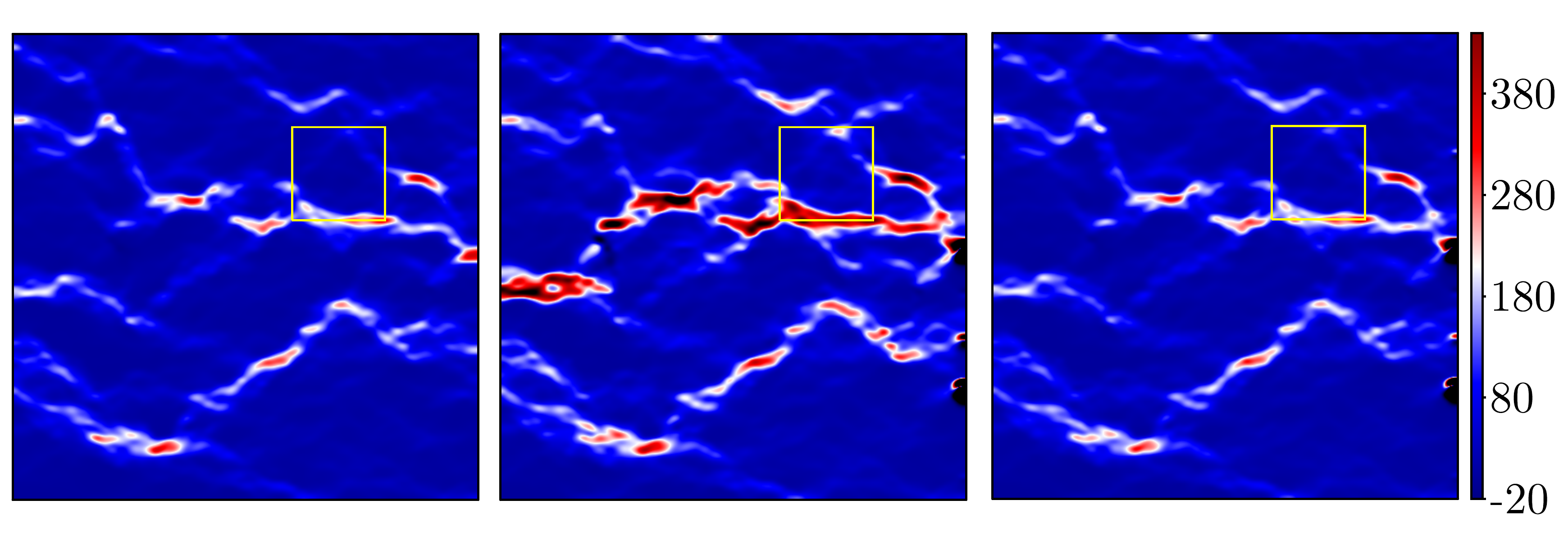}
  	   \includegraphics[width=0.9\textwidth]{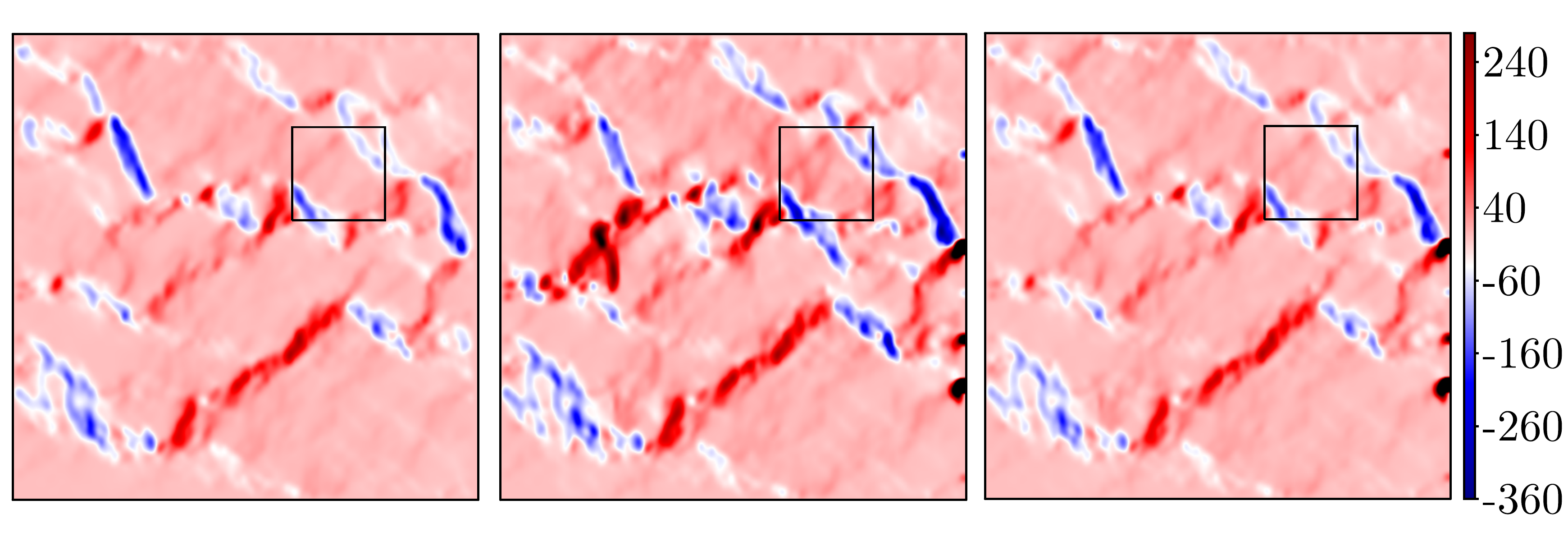} \hspace*{-0.9cm}
      \includegraphics[width=0.25\textwidth]{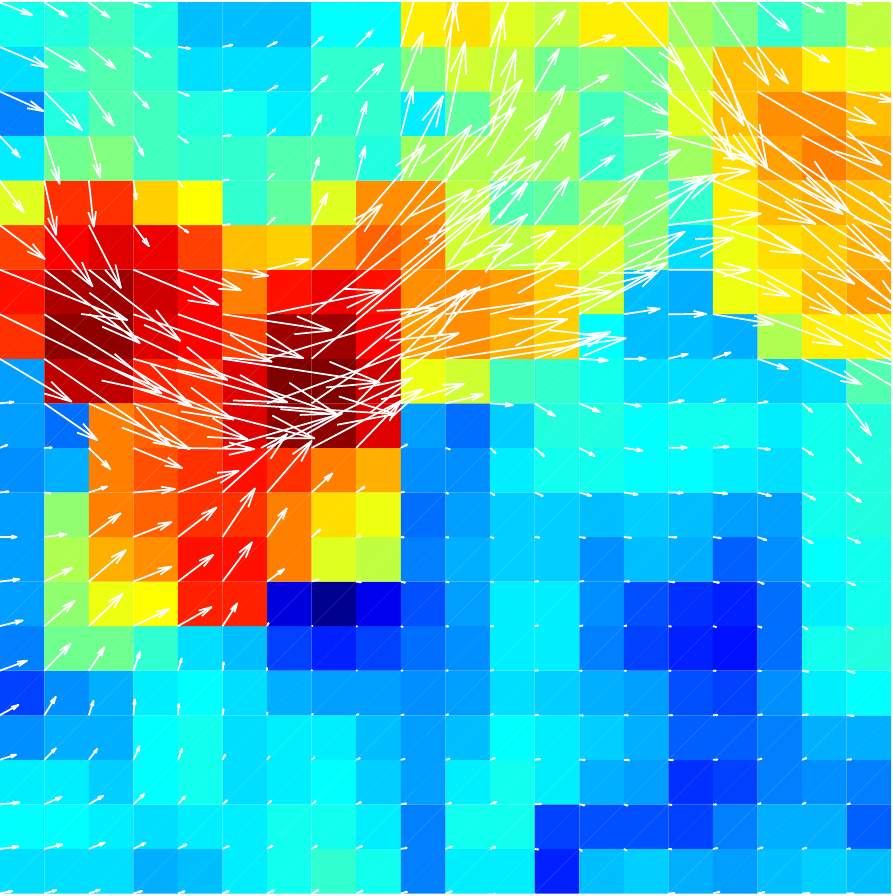}\hspace*{0.3cm}
      \includegraphics[width=0.25\textwidth]{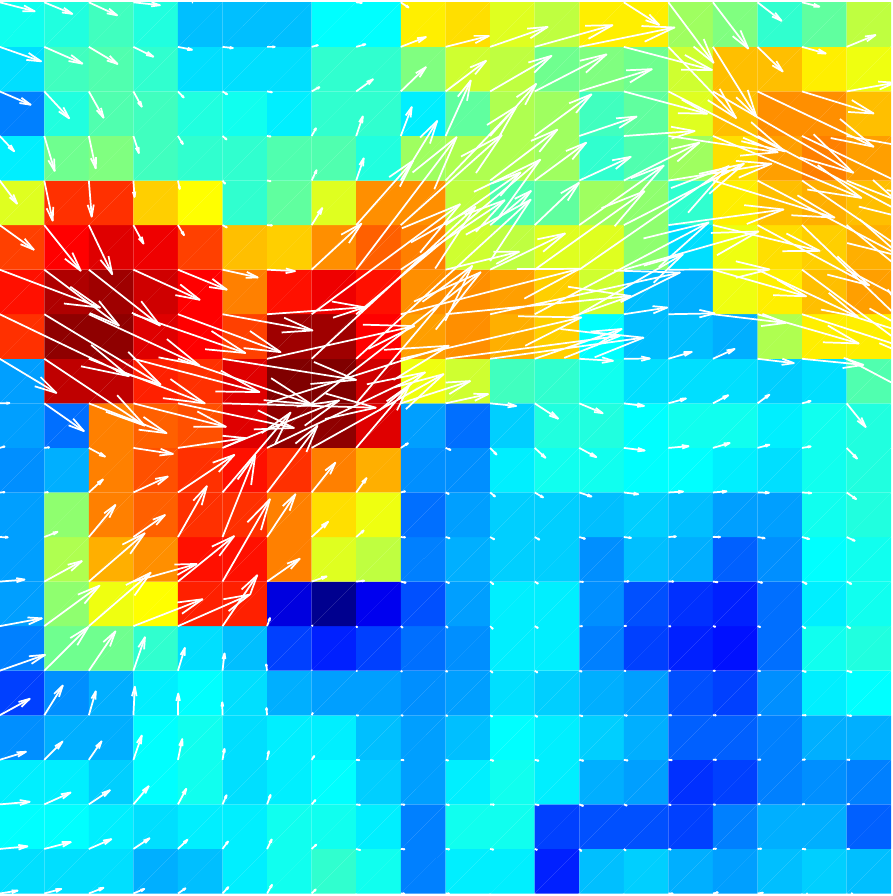}\hspace*{0.3cm}
      \includegraphics[width=0.25\textwidth]{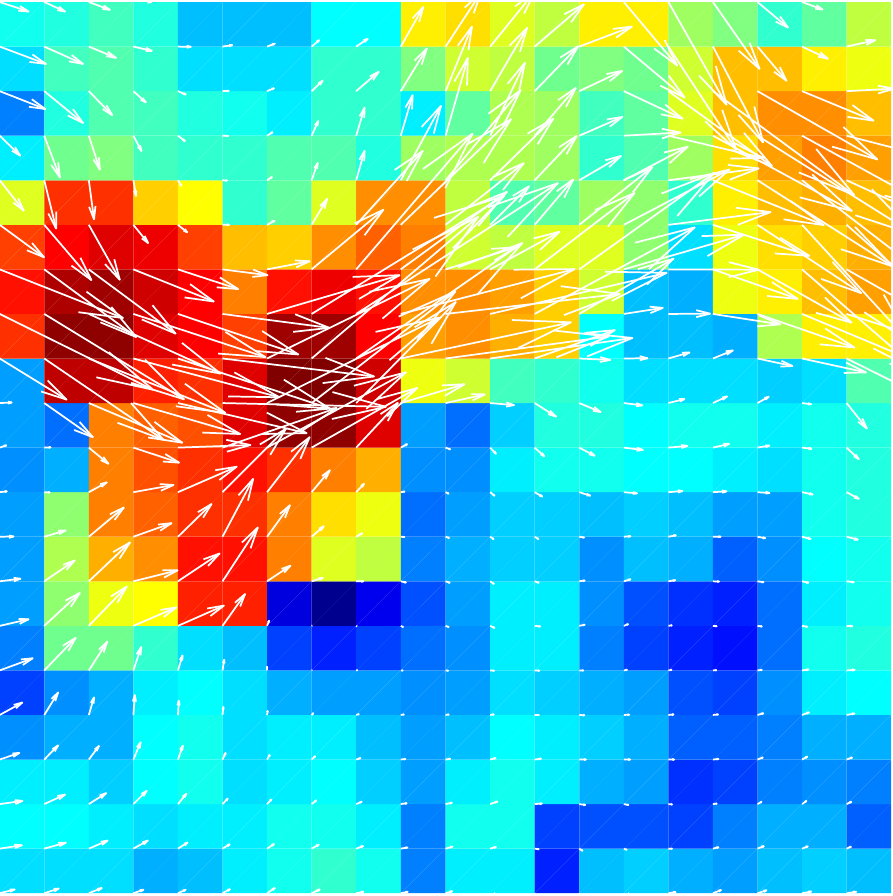}
  \end{center}
  \vspace*{-0.7cm}
  \caption{Velocity computed using three methods. The top two plots show the velocity profile on the whole domain using random permeability with the reference on the left, $L_z=1$ in the center, and $L_z=4$ on the right. The bottom set of plots shows the velocity on the region shown in black and yellow on the top two plots.}
  \label{fig:quiver_stan}
\end{figure}

\subsection{Two-phase convergence}
In solving Eq.~\eqref{eq:sat_eqn} we use quadratic relative permeability curves $k_{rw} = S^2$ and $k_{ro} = (1-S)^2$, along with $\mu_w = 1$ and  $\mu_o = 5$ for the fluid viscosities. For the initial condtion, the value at the left edge is set as $S = 1$ and we assume $S(x,0) = 0$ elsewhere. 
Results for application of the method to two-phase flow are shown in Figs.~\ref{fig:tp_hc} and \ref{fig:tp_stan} using the same permeabilities from Fig.~\ref{fig:perm}. Each of the plots shows the reference saturation in the first row at three time levels, $L_z=1$ on the second row at the same three time levels, $L_z=2$ in the third row, and $L_z=4$ in the last row. The improvement in the deterministic case is significant as can be seen in the Fig.~\ref{fig:tp_hc} and verified with the relative error shown in Fig.~\ref{fig:sat_error}. And as expected, for the random field we see a noticeable (but less pronounced) error decline from the profiles in Figs.~\ref{fig:tp_stan}, and the relative error plots in Fig.~\ref{fig:sat_error}. We mention that in Figs.~\ref{fig:tp_hc} and \ref{fig:tp_stan}, the solutions corresponding to the case when $L_z = 4$ are essentially indistinguishable from the fine scale reference solutions. We reiterate that the size of the resulting linear systems from the pressure equation in Eq.~\eqref{eq:p_eqn} will depend on $L_z$ and the coarse mesh size. The system for the reference cases are of size $10201^2$ for the deterministic permeability and $14641^2$ for the random permeability. Since the coarse mesh is $10 \times 10$ for both cases, using $L_z=1, L_z=2,$ and $L_z=4$ results in linear systems of size $121^2$, $202^2,$ and $364^2$, respectively.

\begin{figure}[htb]
  \begin{center}
      \includegraphics[width=0.9\textwidth]{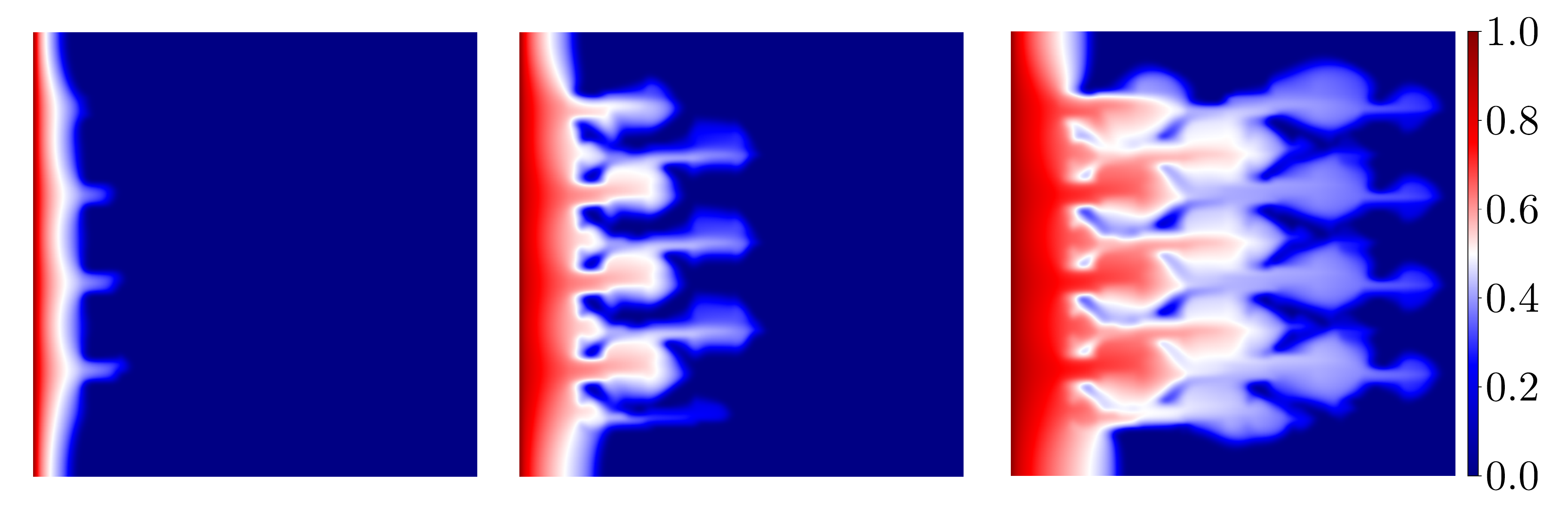}
      \includegraphics[width=0.9\textwidth]{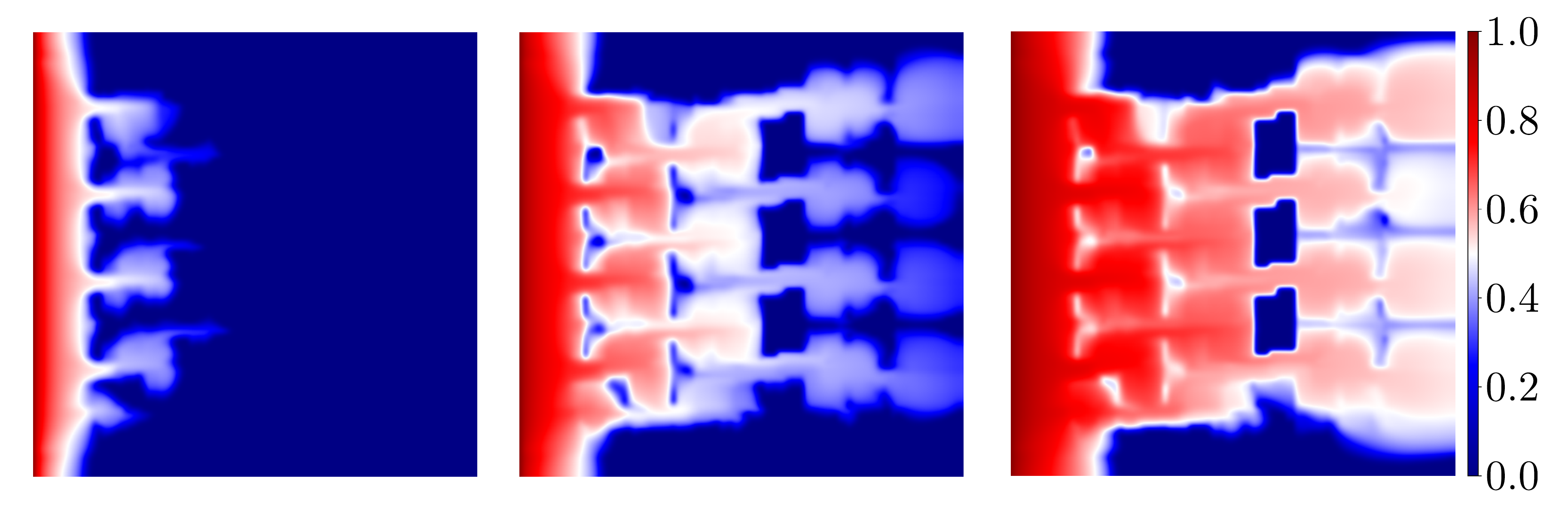}
      \includegraphics[width=0.9\textwidth]{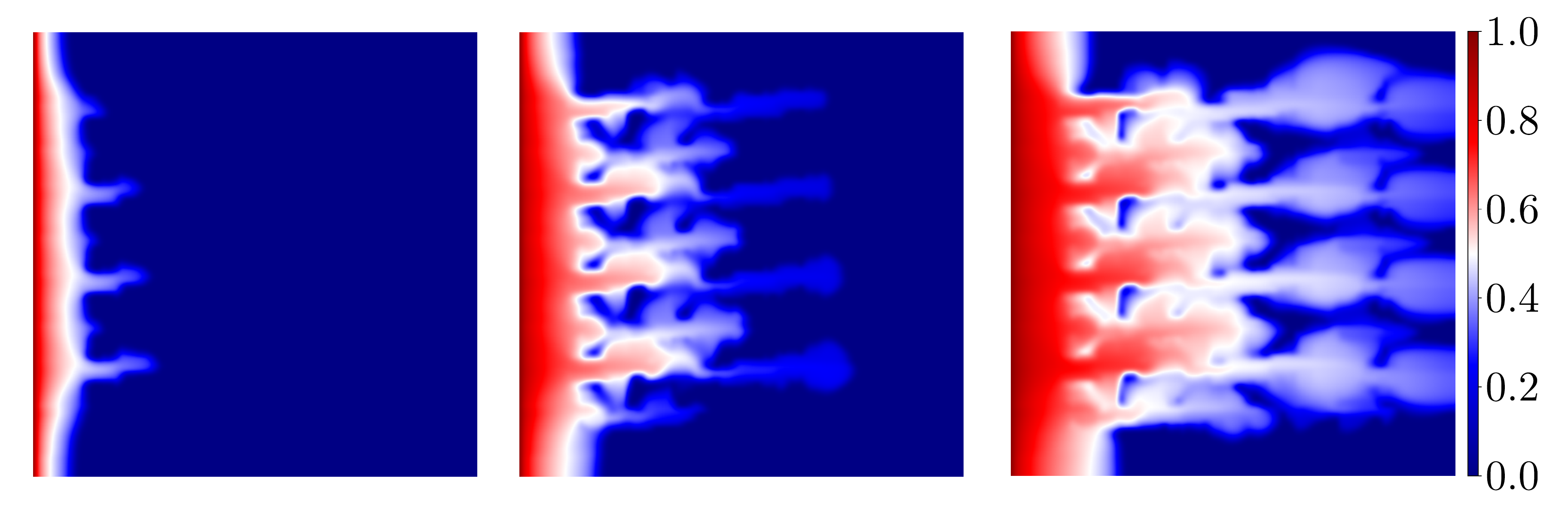}
      \includegraphics[width=0.9\textwidth]{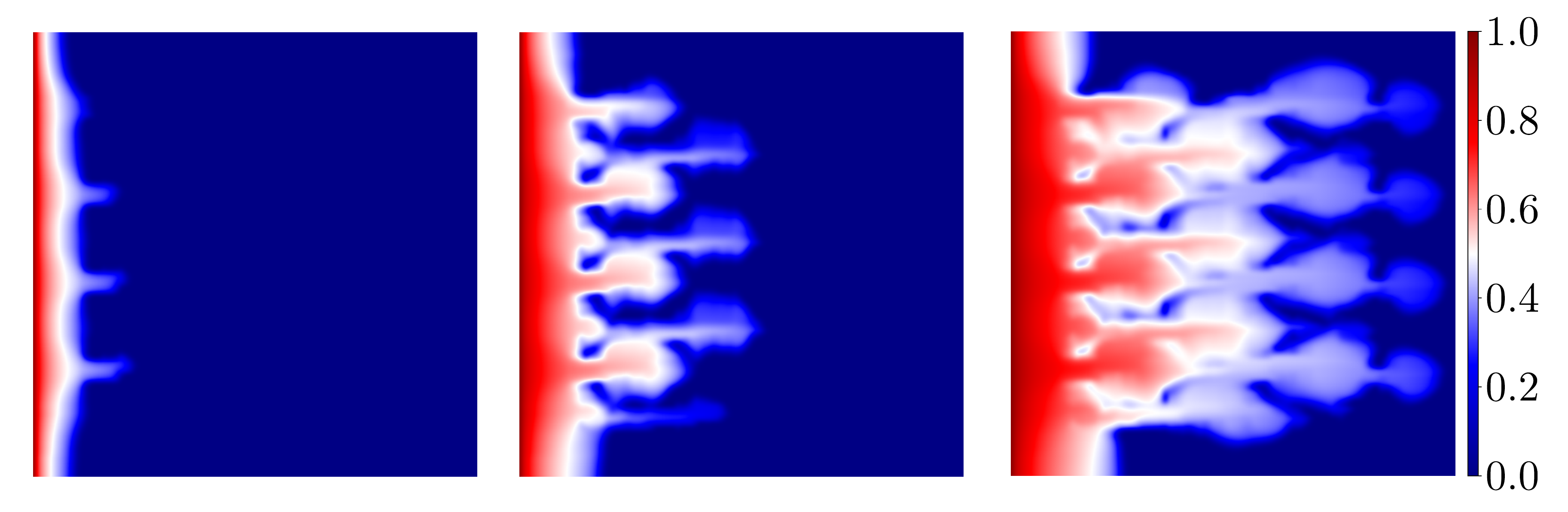}
  \end{center}
  \vspace*{-0.7cm}
  \caption{Deterministic permeability. The reference saturation is shown in the top row at three different time levels. The second through fourth row are $L_z=1, L_z=2$, and $L_z=4$ respectively.}
  \label{fig:tp_hc}
\end{figure}

\begin{figure}[htb]
  \begin{center}
      \includegraphics[width=0.9\textwidth]{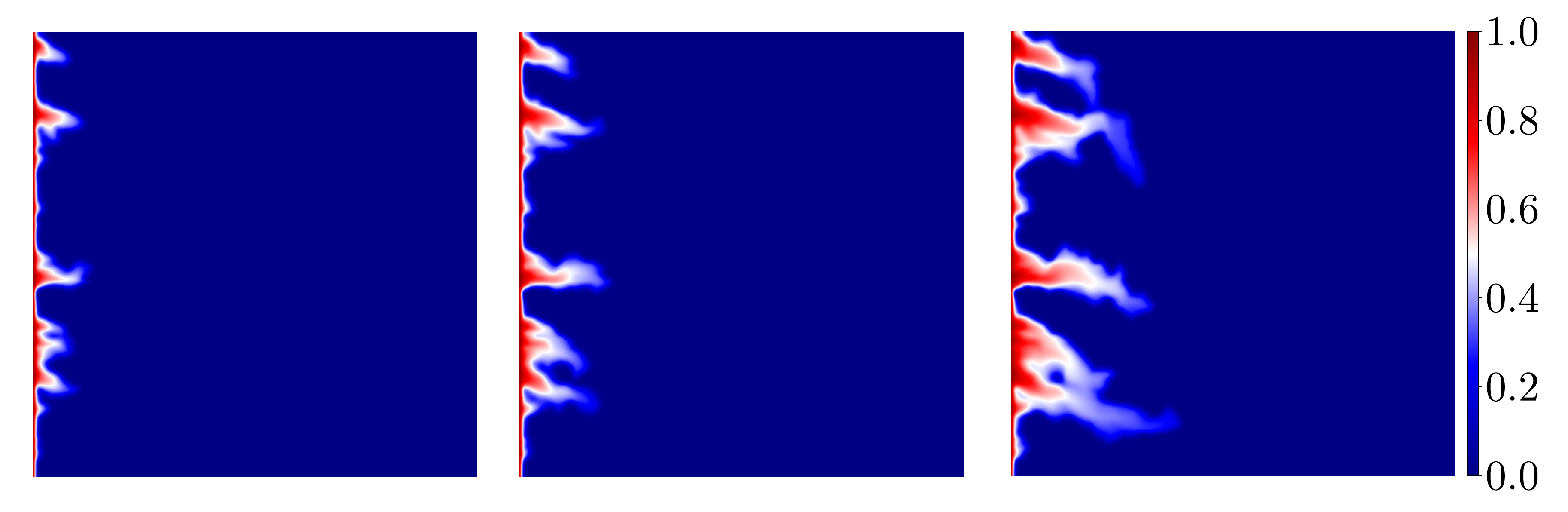}
      \includegraphics[width=0.9\textwidth]{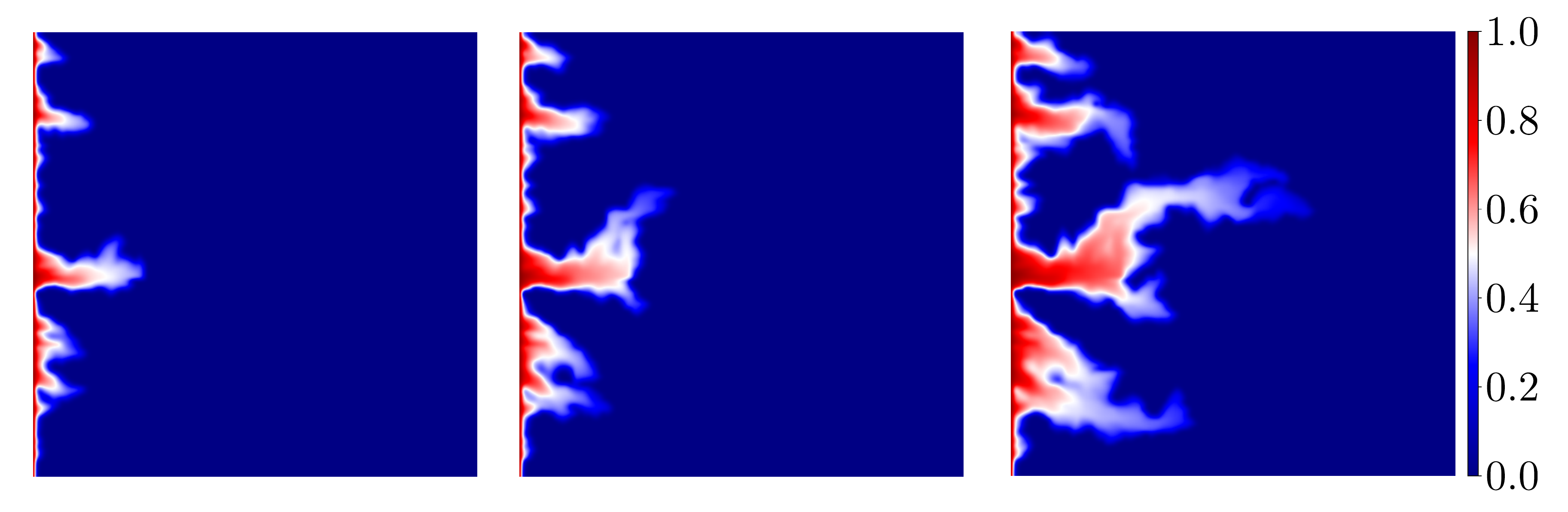}
      \includegraphics[width=0.9\textwidth]{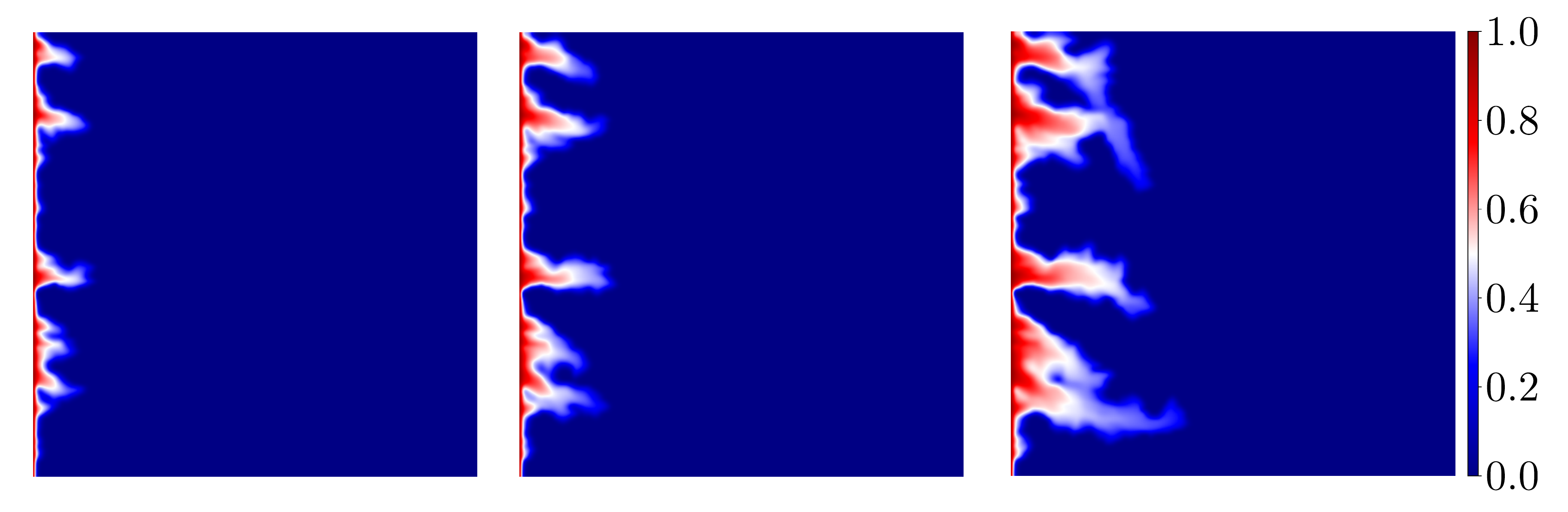}
      \includegraphics[width=0.9\textwidth]{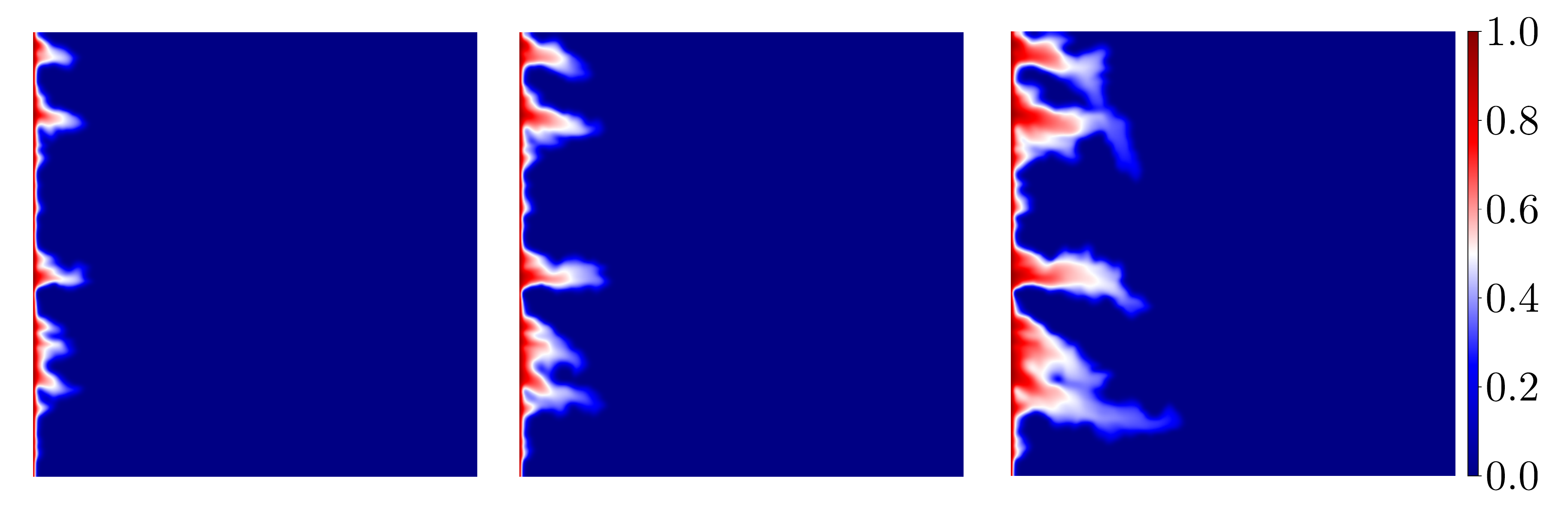}
  \end{center}
  \vspace*{-0.7cm}
  \caption{Random permeability. The reference saturation is shown in the top row at three different time levels. The second through fourth row are $L_z=1, L_z=2$, and $L_z=4$ respectively. }
  \label{fig:tp_stan}
\end{figure}

The $L^2$ error of the saturation profiles offered in Figs.~\ref{fig:tp_hc} and \ref{fig:tp_stan} is presented in Fig.~\ref{fig:sat_error}. As was evident from the velocity results, the error is improved by increasing the number of enrichment functions up to a certain threshhold, and then the reduction is minimal as more functions are added. This suggests the number of enrichment functions should be chosen to minimize the overhead in the calculation. As seen above, the size of the linear system to be solved grows quickly as the number of functions is increased and, at some point, there is little to be gain by adding to the enrichment.

\begin{figure}[htb]
  \begin{center}
      \includegraphics[width=0.45\textwidth]{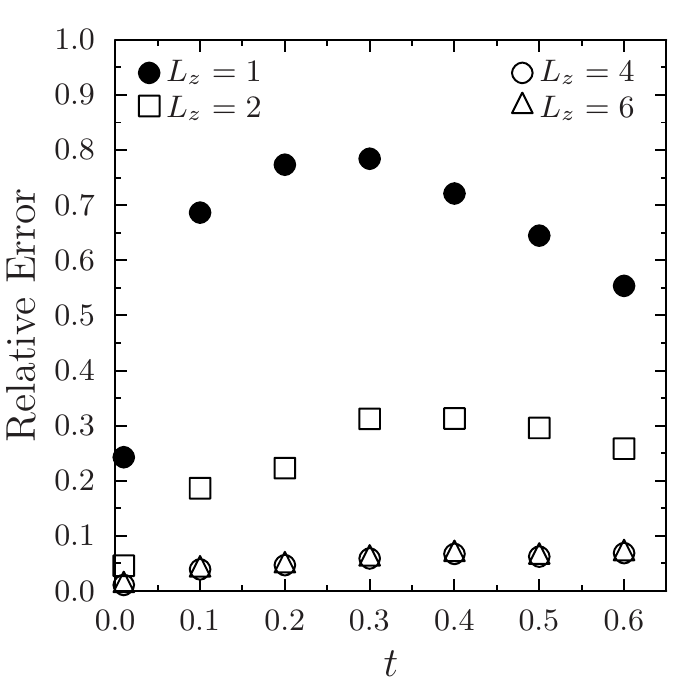} \hspace*{0.6cm}
      \includegraphics[width=0.45\textwidth]{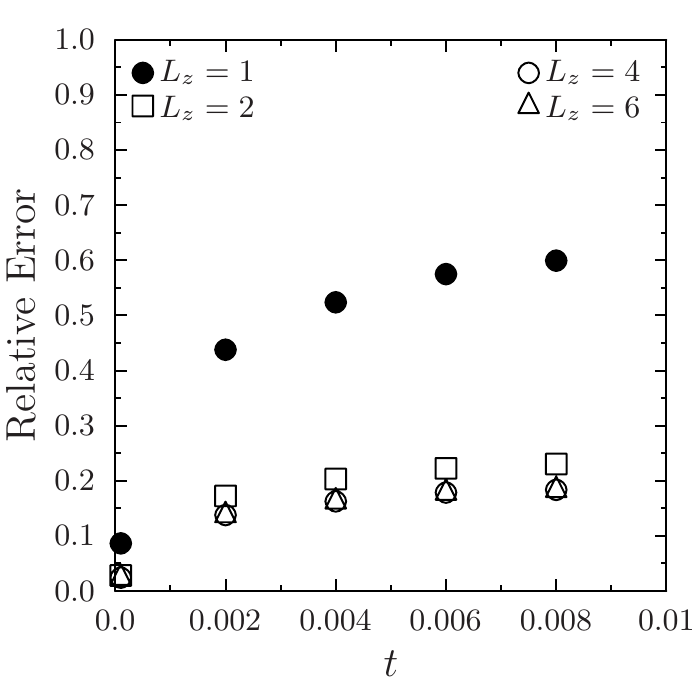}
  \end{center}
  \vspace*{-0.7cm}
  \caption{Comparison of the $L^2$ error of the saturation for deterministic (left) and random (right) as a function of time.}
  \label{fig:sat_error}
\end{figure}

For a comparison on the mesh size, the deterministic permeability was used to simulate two-phase flow with $L_z=4$ using a $10 \times 10$ and $25 \times 25$ coarse mesh. Analogous results are presented for the random field, except that a $30 \times 30$ coarse mesh was used for the refinement. The results are shown in Fig.~\ref{fig:h_error}. As expected, the error is reduced with a finer coarse mesh at the same level of enrichment. This is due to the fact that the coarse mesh is finer and the behavior of the permeability is captured to an acceptable degree by fewer basis functions.

\begin{figure}[htb]
  \begin{center}
	  \includegraphics[width=0.45\textwidth]{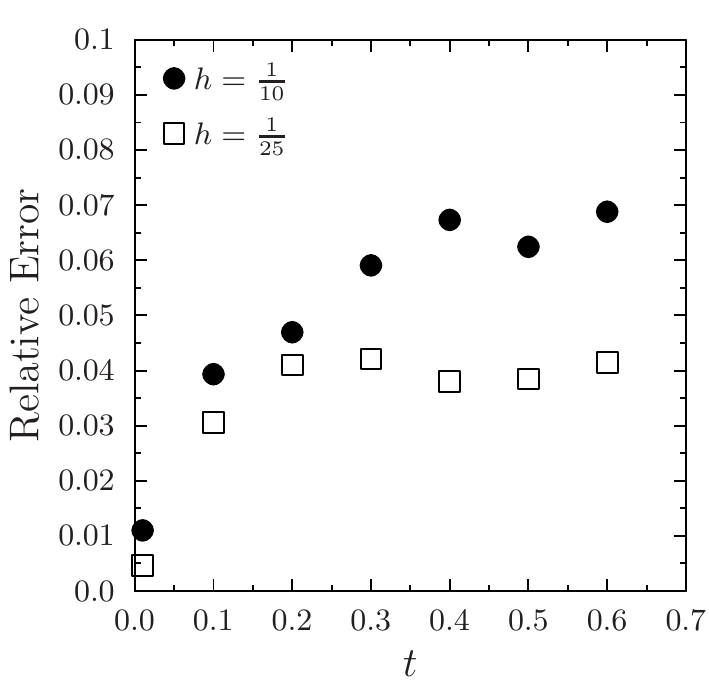} \hspace*{0.6cm}
      \includegraphics[width=0.45\textwidth]{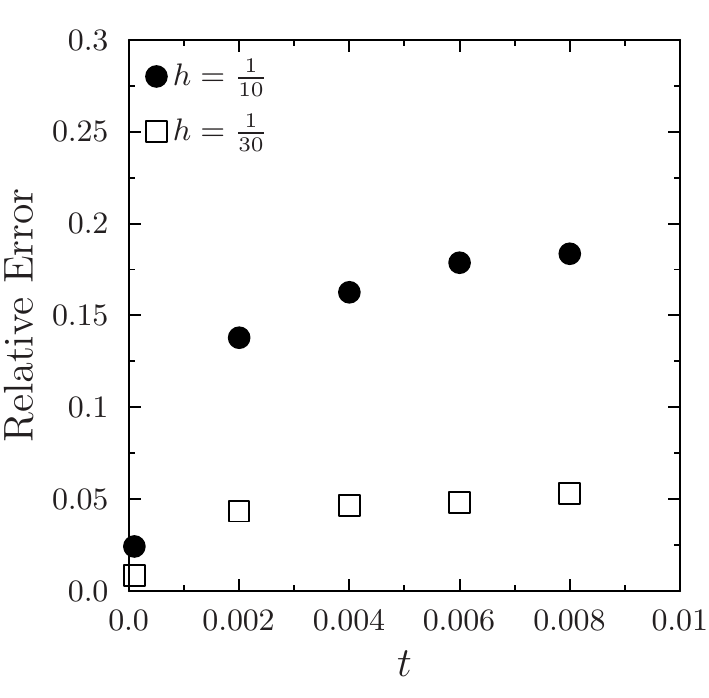}  
  \end{center}
  \vspace*{-0.7cm}
  \caption{Error comparison as the coarse mesh is changed with $L_z=4$ and both permeabilities. Filled circles show the error behavior as a function of time with a $10 \times 10$ coarse mesh. Squares show the error behavior with a 
finer coarse mesh. }
  \label{fig:h_error}
\end{figure}

\section{Concluding remarks}
\label{conclusion}
In this paper we present a method for the construction of locally conservative flux fields from GMsFEM pressure solutions. The method hinges on an element-based postprocessing procedure in which an independent set of $4 \times 4$ linear systems need to be solved in order to compute the flux values. In order to test the performance of the method we consider two permeability coefficients that exhibit distinct classes of heterogeneity. In addition, we apply the proposed method a two-phase flow model in which the flux values are coupled to a hyperbolic transport equation. The increase in accuracy associated with the computation of the GMsFEM pressure solutions is inherited by the associated flux fields and saturation solutions, and is shown to closely depend on the size of the reduced-order systems. In particular, the addition of more basis functions to the enriched, multiscale solution space yields solutions that more accurately capture the behavior of the fine scale model. In the future we wish to address related techniques in which the computation of conservative flux fields is built within the framework of generlized multiscale finite volume element methods.


\section*{References}
\bibliographystyle{siam}
\bibliography{paper2}


\end{document}